\documentclass{article}
\usepackage[ansinew]{inputenc}
\usepackage[english]{babel}
\usepackage{latexsym}
\usepackage{amsthm}
\usepackage{amsmath}
\usepackage{amsfonts}
\usepackage{amssymb}
\usepackage{graphicx}
\usepackage{times}
\usepackage{url}
\usepackage{float}

\newcommand{\AAAcero}{\left[
\begin{array}{ccc}
a_{11}-a_{31}&a_{12}-a_{32}&a_{13}-a_{33}\\a_{21}-a_{31}&a_{22}-a_{32}&a_{23}-a_{33}\\0&0&0\\
\end{array}\right]}

\newcommand{\LLL}{\left[
\begin{array}{ccc}
0&-d-d_2&-2d-d_3\\
-2d-d_1&0&-d-d_3\\
-d-d_1&-2d-d_2&0\\
\end{array}\right]}

\newcommand{\LLLcero}{\left[
\begin{array}{ccc}
d+d_1&d&-2d-d_3\\
-d&2d+d_2&-d-d_3\\
0&0&0\\
\end{array}\right]}

\newcommand{\LLLTmenoscero}{\left[
\begin{array}{ccc}
-2d-d_3&d+d_1-d_3&d+d_1\\
-d+d_2-d_3&-d-d_3&2d+d_2\\
0&0&0
\end{array}\right]}

\newcommand{\UUU}{\left[
\begin{array}{ccc}
0&-2d-d_2&-d-d_3\\
-d-d_1&0&-2d-d_3\\
-2d-d_1&-d-d_2&0\\
\end{array}\right]}

\newcommand{\FH}{\left[
\begin{array}{ccc}
0&-d-d_2&-2d-d_3-h_3\\
-2d-d_1-h_1&0&-d-d_3-g\\
-d-d_1&-2d-d_2-h_2&0\\
\end{array}\right]}

\newcommand{\GH}{\left[
\begin{array}{ccc}
0&-d-d_2-g_2&-2d-d_3-h_3\\
-2d-d_1-h_1&0&-d-d_3-g_3\\
-d-d_1-g_1&-2d-d_2-h_2&0\\
\end{array}\right]}

\newcommand{\FHcero}{\left[
\begin{array}{ccc}
d+d_1&d+h_2&-2d-d_3-h_3\\
-d-h_1&2d+d_2+h_2&-d-d_3-g\\
0&0&0\\
\end{array}\right]}

\newcommand{\FHTmenoscero}{\left[
\begin{array}{ccc}
-2d-d_3-h_3&d+d_1-d_3+h_1-g&d+d_1\\
-d+d_2-d_3-h_3&-d-d_3-g&2d+d_2+h_2\\
0&0&0\\
\end{array}\right]}

\newcommand{\suces}[3]{{#1_{#2},\hdots,#1_{#3}}}

\newcommand{\su}[1]{#1_1,#1_2,#1_3}

\newcommand{\co}[3]{\left[\begin{array}{c}
#1\\#2\\#3
\end{array}\right]}

\newcommand{\est}{\operatorname{st}}
\newcommand{\col}{\operatorname{col}}
\renewcommand{\span}{\operatorname{span}}
\newcommand{\cospan}{\operatorname{co--span}}
\newcommand{\row}{\operatorname{row}}
\newcommand{\const}{\operatorname{const}}
\newcommand{\ant}{\operatorname{ant}}
\newcommand{\soma}{\operatorname{soma}}
\newcommand{\diag}{\operatorname{diag}}

\newcommand{\realamp}{\R\cup\{-\infty\}}
\newcommand{\realampp}{\R\cup\{+\infty\}}

\newcommand{\m}{\medskip}

\newcommand{\N}{\mathbb{N}}

\newcommand{\R}{\mathbb{R}}
\renewcommand{\P}{\mathbb{P}}
\newcommand{\T}{\mathbb{T}}

\newcommand{\CC}{\mathcal{C}}

\newcommand{\TT}{\mathcal{T}}

\newcommand{\oa}{\overline{a}}
\newcommand{\ob}{\overline{b}}
\newcommand{\oc}{\overline{c}}
\newcommand{\od}{\overline{d}}
\renewcommand{\oe}{\overline{e}}
\newcommand{\of}{\overline{f}}

\newcommand{\id}{\operatorname{id}}
\newcommand{\im}{\operatorname{im}}

\newtheorem{thm}{Theorem}
\newtheorem{lem}{Lemma}
\newtheorem{dfn}{Definition}
\newtheorem{rem}{Remark}
\newtheorem{ex}{Example}
\newtheorem{cor}{Corollary}

\begin{document}

\title{Tropical linear maps on the plane}

\author{M.J. de la
Puente\thanks{Departamento de Algebra, Facultad de Matem{\'a}ticas,
Universidad Complutense, 28040--Madrid, Spain,
mpuente@mat.ucm.es}\ \thanks{Partially
supported by UCM research group 910444.}}

\maketitle

\date{}
\begin{abstract}



In this paper we  fully describe  all \emph{tropical linear maps in the tropical projective plane $\T\P^2$}, that is, maps from the tropical plane to itself given by tropical multiplication by a  real $3\times 3$ matrix $A$.
The map $f_A$ is  continuous and piecewise--linear  in the classical sense. In some particular cases, the map $f_A$ is a parallel projection onto the set spanned by the columns of $A$. In the general case, after a change of coordinates, the map collapses at most three regions of the plane onto certain  segments, called antennas, and is  a parallel projection elsewhere (theorem \ref{thm:f_A}).

\m In order to study  $f_A$, we may assume that $A$ is normal, i.e., $I\le A\le 0$, up to  changes of coordinates. A given matrix $A$ admits infinitely many normalizations. Our approach is to define and compute a unique normalization   for $A$  (which we call \emph{lower canonical normalization}) (theorem \ref{thm:canonical_nf}) and then always work with it, due both to its algebraic simplicity and its geometrical meaning.

\m On  $\R^n$, any $n\in\N$,  some aspects of tropical linear maps have been studied in \cite{Butkovic_S}. We work in $\T\P^2$, adding a geometric view and doing everything explicitly. We give precise pictures.

\m Inspiration for this paper comes from \cite{Ansola_tri,Barbero,Butkovic_S,Cohen,Develin,Yoeli}. We have tried to make it self--contained. Our preparatory   results  present   noticeable relationships between the algebraic properties of a given matrix  $A$ (idempotent normal matrix, permutation matrix, etc.) and  classical geometric properties  of the points spanned by the columns of $A$   (classical convexity and others); see theorem \ref{thm:normalZ=0} and corollary \ref{cor:convex}. As a by--product, we compute all the tropical square roots of normal  matrices of a certain type; see corollary \ref{cor:sqrt}. This is, perhaps, a curious result  in  tropical algebra.
Our final aim is, however, to give a precise  description of the  map $f_A:\T\P^2\to\T\P^2$. This  is particularly easy  when two tropical triangles arising from $A$ (denoted $\TT_A$ and $\TT^A$)  fit as much as possible. Then the action of $f_A$ is easily described on (the closure of) each cell of
 the cell decomposition  $\CC^A$; see  theorem \ref{thm:f_A}.

\m \emph{Normal matrices} play a crucial role in this paper.
The tropical powers of normal matrices of  size $n\in \N$  satisfy  $A^{\odot n-1}=A^{\odot {n}}=A^{\odot n+1}=\cdots$. This statement can be traced back, at least,  to \cite{Yoeli}, and appears later  many times,  such as \cite{Akian_HB, Akian_AMS, Butkovic_S, Cuninghame,Cuninghame_New}. In lemma \ref{lem:normal}, we  give a direct proof  of this fact, for $n=3$. But now the equality $A^{\odot2}=A^{\odot3}$ means that the columns of $A^{\odot 2}$ are three fixed points of $f_A$ and, in fact, \emph{any point spanned by the columns of $A^{\odot2}$ is  fixed by $f_A$}.
Among $3\times 3$ normal matrices, the \emph{idempotent} ones (i.e., those satisfying $A=A^{\odot 2}$) are particularly  nice:
we prove that the columns of such a matrix  tropically span a  set which is classically compact, connected and \emph{convex} (lemma \ref{lem:A=A2} and corollary \ref{cor:convex}).
In our terminology, it is a \emph{good tropical triangle}.

\end{abstract}

\section{Introduction, Notations and Background on Tropical Mathematics}

Many results on  finite dimensional tropical linear algebra (spectral theory, etc.) have been published over the last 40 years and more; they are summarized in \cite{Akian_HB, Cuninghame_New, Gaubert}, where a wide bibliography can also be found. Two recent up--to-date collections of papers are \cite{Litvinov_ed, Litvinov_ed_2}
In this  paper we will use the adjective \emph{classical} as opposed to \emph{tropical}. Most definitions in tropical mathematics just  mimic the classical ones.
However, tropical geometry is a peculiar one.   Say an inhabitant of the tropical plane is disoriented. He/she takes a look at a compass   and tries to spot the tropical cardinal points. There are only three: \emph{east, north and south--west}! Accordingly, he/she will set the positive part of the three coordinate axes in the given directions,  when doing geometry on the plane. He/she will find out that a  generic tropical line in the tropical plane looks like a tripod (it has a vertex!) although some particular tropical lines look just like classical lines, see figure \ref{fig:recta_1}.
\begin{figure}[H]
 \centering
  \includegraphics[width=8cm,keepaspectratio]{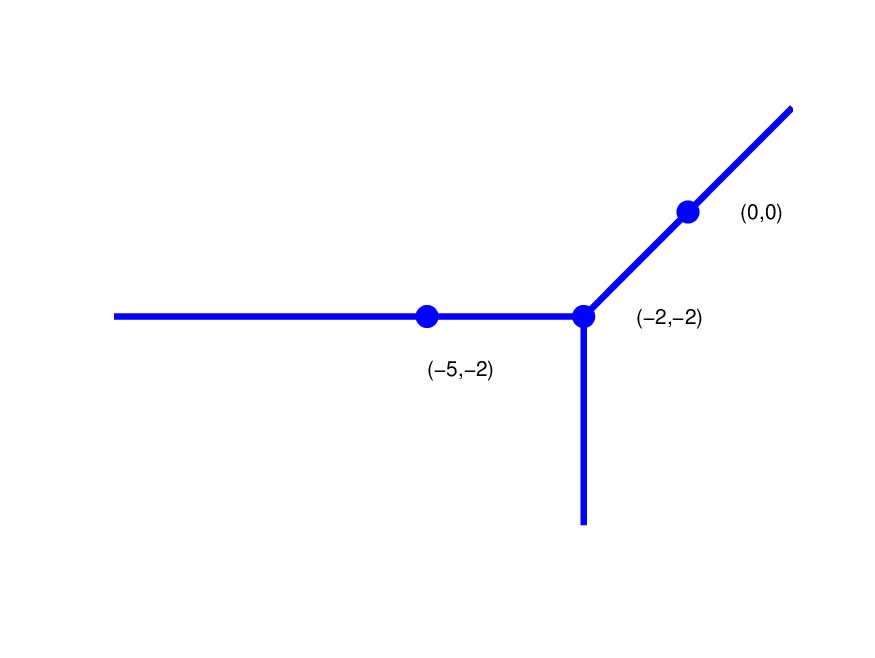}\\
  \caption{Tropical line with vertex at the point $(-2,-2)$.}
  \label{fig:recta_1}
  \end{figure}

If we happen to go down--town in  a city designed by a  tropical geometer, we will
  find out that the shape of most  blocs   is that of a classical hexagon, with parallel opposite sides of slopes $0,1,\infty$, see figure \ref{fig:downtown_3}.
\begin{figure}[H]
 \centering
  \includegraphics[keepaspectratio,width=6cm]{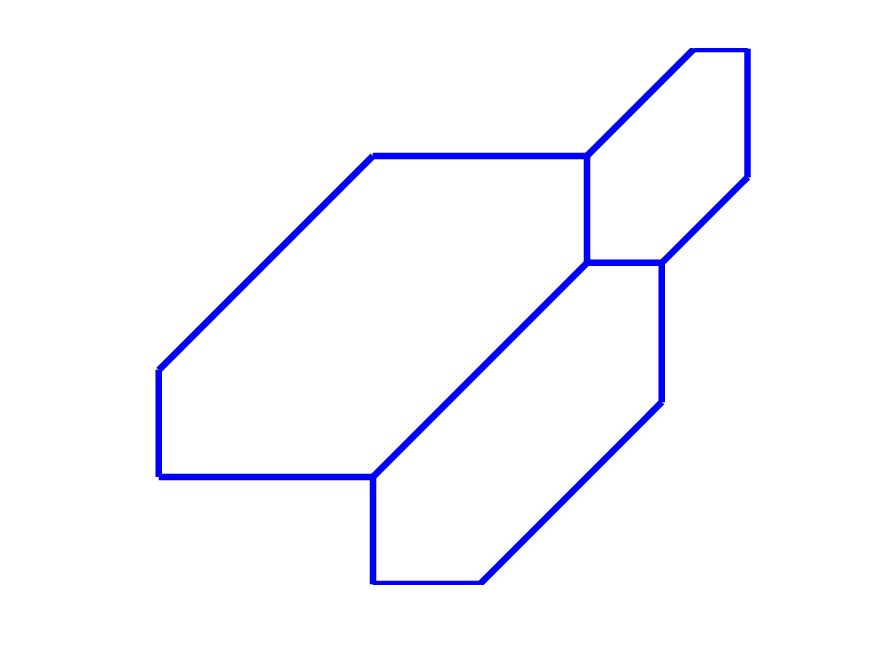}\\
  \caption{Downtown blocs in a tropical city.}
  \label{fig:downtown_3}
  \end{figure}

The shortest path between two given points  is made up of, at most,  two classical segments with slopes $0,1,\infty$. Moreover, the distance between the given  points is the sum of the \emph{integer lengths} (also called \emph{lattice lengths})
of these segments. For instance, the integer length between the points $(-2,-2)$  and $(0,0)$ is $2$ (not $2\sqrt{2}$!)  and the integer length between the points $(-5,-2)$  and $(0,0)$ is $3+2=5$; see figure \ref{fig:recta_1}.
This is, indeed, a sort of Manhattan distance.

So, plane tropical geometry is a funny looking
piecewise--linear geometry. And, by the way, why is it called tropical? Well, the explanation  appears in \cite{Gathmann, Itenberg}, etc. and we must add that some other names have also been used (for this or akin mathematics):  max--plus, dioids, path algebra, extremal algebra, idempotent mathematics, etc.

\m Consider the set $\realamp$ endowed with
tropical addition $\oplus$ and tropical multiplication $\odot$, where
these operations are defined as follows:
    $$a\oplus b=\max\{a,b\},\qquad a\odot b=a+b,$$ for
    $a,b\in\realamp$. Here, $-\infty$ is the neutral element for tropical addition and   $0$ is the neutral element for tropical multiplication.
    Notice that
    $a\oplus a=a$, for all $a$, i.e., tropical addition is \emph{idempotent}. Notice also that $a$ has no inverse with respect to $\oplus$.

\m We will work with $\realamp$, which will be denoted $\T$ and will be called the \emph{tropical semi--field}. We will write $\oplus$ or $\max$, (resp. $\odot$ or $+$) at our convenience.

\m In classical mathematics, we have  a choice in  geometry: affine or projective.  The  \emph{tropical affine plane} is  $\T^2$,
where addition and multiplication are defined
coordinatewise.  In the space
$\T^{3}\setminus\{(-\infty,-\infty,-\infty)\}$ we define
an equivalence relation $\sim $ \label{asser:equiv} by letting $(\su p)\sim (\su q)$ if there exists
$\lambda\in\R$ such that
    $$\lambda\odot(\su p)=(\lambda+ p_1,  \lambda+ p_2, \lambda+ p_{3})=
    (\su q).$$
The equivalence class of $(\su p)$ is denoted
$[\su p]$.  The \emph{tropical projective plane} is the set,
$\T\P^2$, of such equivalence classes. Notice
that, at least, one of the coordinates of any point in
$\T\P^2$ must be finite; see figure \ref{fig:proj_plane_3}.

\begin{figure}[H]
 \centering
  \includegraphics[width=7cm,keepaspectratio]{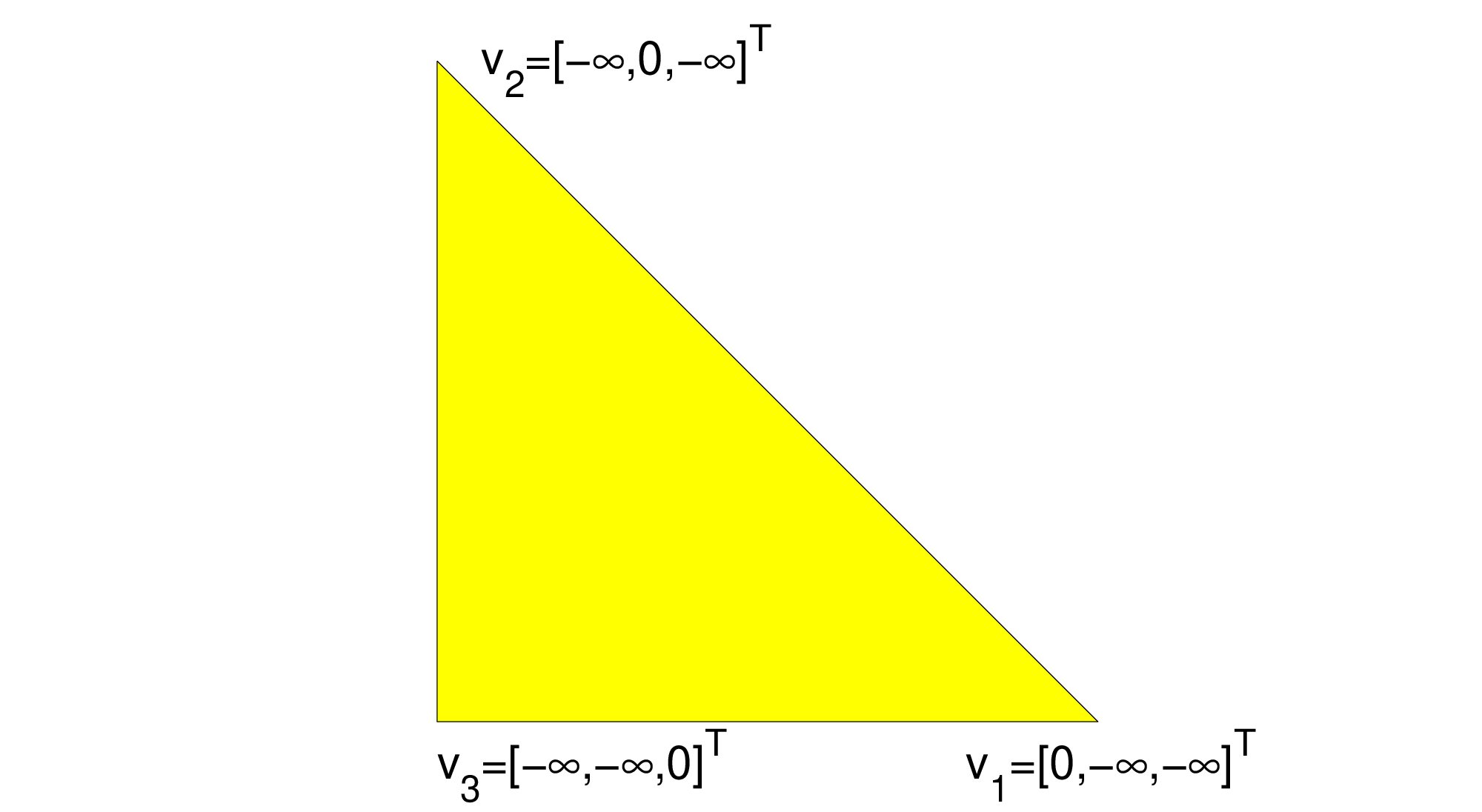}\\
  \caption{Tropical projective plane.}
  \label{fig:proj_plane_3}
  \end{figure}

We endow the tropical plane (either  affine or projective)  with the topology induced by the \emph{Euclidean topology}; the  closure $\overline S$ of a set $S$  refers to this topology. In p. \pageref{dfn:norm} below, we also define a  \emph{tropical norm} in the  projective tropical plane. This  norm gives rise to the Euclidean topology.

 It can be easily proved that $\T\P^2$ is compact. \label{asser:compact} $\T\P^2$ is a \emph{compactification} of  $\T^2$ (and also of $\R^2$; see p. \pageref{asser:compact2}). Indeed, the image of the  injective map $\varphi:\T^2\to\T\P^2$ given by $(x,y)\mapsto [x,y,0]$ is open and dense.  The  \emph{boundary points} are those $[x,y,-\infty]$ in $\T\P^2$.
  In fact, $\T\P^2$ is homeomorphic to a classical  triangle in $\R^2$ (the vertices of $\T\P^2$ are $v_1=[0,-\infty,-\infty]^T, v_2=[-\infty,0,-\infty]^T$ and $v_3=[-\infty,-\infty,0]^T$; see figure \ref{fig:proj_plane_3}).

 Now, for any $p=[x,y,z]$, we have $\varphi^{-1}(p)=(x-z,y-z)$,  whenever $z\neq -\infty$. Taking $(x-z,y-z,0)$ as a representative of $p$ will be expressed by saying that   \emph{we work in $Z=0$}. In other words, to work in $Z=0$ it is just a way of  passing from the projective to the  affine tropical plane.

\m The simplest objects in the tropical plane are lines. Given
 a tropical linear form
$$p_1\odot X\oplus p_2\odot Y\oplus p_3\odot Z=\max\{p_1+X ,
    p_2+Y,p_3+Z
     \}$$ a \emph{tropical line} consists  of the points  $[x,y,z]\in\T\P^2$ where \emph{the maximum is attained, at least, twice}, (this is the tropical analog of the classical vanishing point set). Denote this line by $L_p$, where $p=[p_1,p_2,p_3]\in\T\P^2$.

 \m Most lines in the  tropical plane look like \emph{tripods}. Indeed, if two coefficients are equal to $-\infty$, then $L_p$ is a boundary component of $\T\P^2$. If  $p_j=-\infty$ for just one $j$ then, in $Z=0$, $L_p$ is nothing but a classical slope--one line.  If all  $p_j$ are real, the $L_p$ is  the union of three rays. The directions of these rays
are west,   south and north--east (just opposite to the cardinal directions of the tropical plane!) and these rays are emanating from the point $-p$, called the \emph{vertex} of $L_p$. The latter is the generic case.

\m Let two points $p,q$ in the tropical plane  be given.  There may exist just one or infinitely many tropical lines passing through $p$ and $q$. In the latter case, just one of these lines can be described as  the limit, as
$\epsilon$ tends to zero, of the tropical lines going
through perturbed points $p^{v_\epsilon}, q^{v_\epsilon}$.
Here, $p^{v_\epsilon}$ denotes a translation of $p$ by a
length--$\epsilon$ vector  $v_\epsilon$,  see
\cite{Gathmann,Richter}. We denote this limit
line by $pq$ and call it the tropical
\emph{stable join}   of $p,q$.

\m Now, given two tropical lines $L_p,L_q$ in the  plane, the
\emph{stable intersection} of $L_p,L_q$, denoted $L_p\cap_{\est}
L_q$, is defined as the limit point, as $\epsilon$
tends to zero, of the intersection of perturbed lines
$L_p^{v_\epsilon}, L_q^{v_\epsilon}$. Here,
$L_p^{v_\epsilon}$ denotes a translation of $L_p$ by a
length--$\epsilon$  vector $v_\epsilon$.

\m There exists a \emph{duality} between lines  and points since
$$q\in L_p\iff p\in L_q,$$  meaning that \emph{the maximum}
$\max\{p_1+q_1,p_2+q_2,p_3+q_3\}$\emph{is attained, at least, twice.}
This duality  transforms
stable join  into stable intersection and conversely,
 i.e.,  $$L_p\cap_{\est}L_q=r\iff pq=L_r,$$ for
 $p,q,r$  in $\T\P^2$.

\m The \emph{tropical version of Cramer's rule}\label{dfn:Cramer} (see \cite{Richter}) goes as follows:
 the stable
intersection of the lines $L_p$ and $L_q$ is the point
$$[\max\{p_2+q_3,q_2+p_3\},
 \max\{p_1+q_3,q_1+p_3\},
    \max\{p_1+q_2,q_1+p_2\}].$$ Since the computation of this point is nothing but
      a \emph{tropical version of the
    cross--product} of the triples $p$ and $q$,  we will
     denote it by $p\otimes q$ (this is not to be mixed up with $p\odot q=p+q$). Notice that $p\otimes q=q\otimes p$.
     In other words,  the \emph{tropical version of Cramer's rule}  in the plane can be written as
     $$L_p\cap_{\est}L_q=p\otimes q\qquad \text{and}\qquad  pq=L_{p\otimes q},$$  by duality.
In particular, \emph{$-(p\otimes q)$ is the vertex  of the  line $pq$}, a crucial fact that we use again and again.

\begin{figure}[H]
 \centering
  \includegraphics[keepaspectratio,width=6cm]{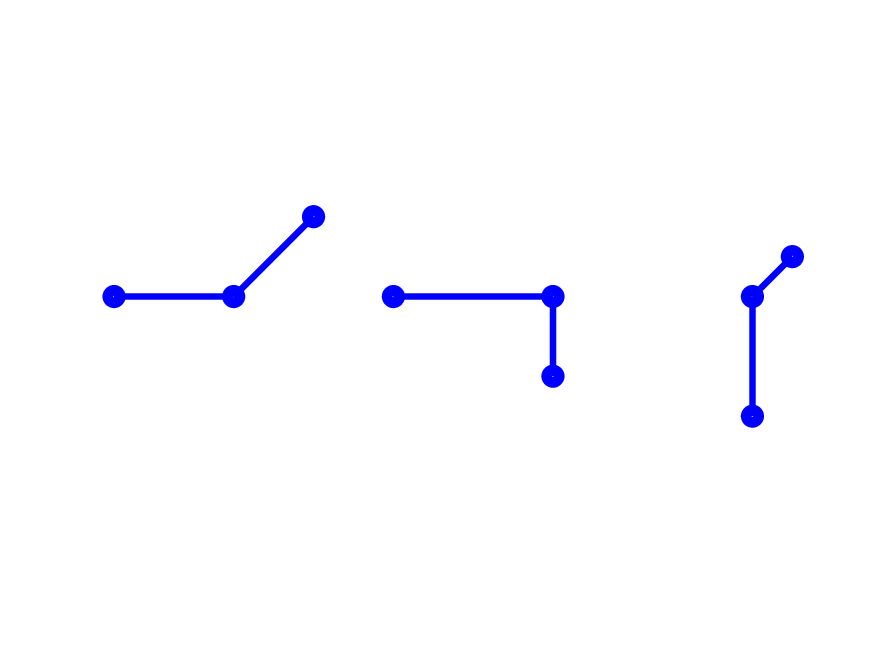}\\
  \caption{Tropical line segments.}
  \label{fig:segmentos_2}
  \end{figure}

Given a  subset $U$ of  points in $\T\P^2$ (resp.  $\T^2$), we can consider the \emph{tropical span of $U$}, denoted $\span(U)$, meaning  the set of points  $u\in\T\P^2$ (resp.  $\T^2$) which can be written as
    $$u=\lambda_1\odot u_1\oplus\cdots\oplus  \lambda_s\odot
     u_s=\max\{\lambda_1+u_1 ,\ldots, \lambda_s+u_s\},$$
     for some $s\in\N$, $\suces u1s\in U$, $\suces {\lambda}1s\in\T$, and not all $\lambda_j$
     equal to $-\infty$ (and $\lambda_1\oplus\cdots\oplus \lambda_s=\max\{\lambda_1,\ldots,\lambda_s\}=0$ when points are in $\T^2$).

\m The \emph{tropical co--span of $U$}, denoted $\cospan(U)$, is the set of points  $u$ which can be written as
    $$u=\min\{\lambda_1+u_1 ,\ldots, \lambda_s+u_s\},$$
     for some $s\in\N$, $\suces u1s\in U$, $\suces {\lambda}1s\in\realampp$, and not all $\lambda_j$
     equal to $+\infty$ (and $\min\{\lambda_1,\ldots,\lambda_s\}=0$ when points are in $\T^2$).

\m
 Given two points $p,q\in\T\P^2$, we know that  $-(p\otimes q)$ represents the vertex of the  line  $pq$. Thus $\span(p,q)$  is the  union of the classical segments $\overline{p,-(p\otimes q)}$ and $\overline{-(p\otimes q),q}$.   Dually, $\cospan(p,q)$ is the union of the classical segments $\overline{p,(-p)\otimes (-q)}$ and $\overline{(-p)\otimes (-q),q}$. It follows that the points $p,-(p\otimes q),q, (-p)\otimes (-q)$ are the vertices of a classical parallelogram, see figure \ref{fig:co_span_4}.
\begin{figure}[H]
 \centering
  \includegraphics[keepaspectratio,width=10cm]{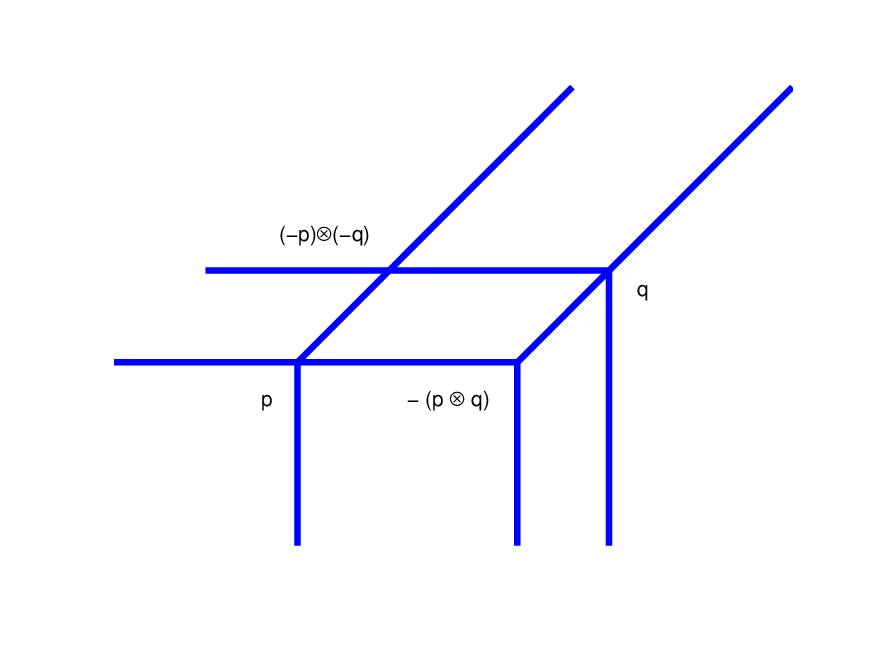}\\
  \caption{Span and co--span of points $p,q$.}
  \label{fig:co_span_4}
  \end{figure}

\m Another sort of duality is taking place here. \label{dfn:max-min_duality} Indeed, we may consider $\realampp$ endowed with tropical addition $\oplus'=\min$ and the same tropical multiplication $\odot$.
The relationship between these two operations is $\max\{p,q\}=-\min\{-p,-q\}$, whence $$p\oplus q=-(-p)\oplus'(-q),$$ for $p,q\in\R$. This \emph{max--min duality} appears  in the literature, see \cite{Baccelli,Cuninghame,Cohen}, etc.

\m  Why do we care about the co--span? A tropical triangle can be determined  by three  points, or by three lines.
First, a tropical triangle  $\TT $ is determined by three points $a,b,c$, meaning
 $$\TT =\span(a,b,c).$$ If the points are tropically collinear then $\TT $ is not generic.

The \emph{sides} of $\TT $  are, by definition, the tropical lines $ab,bc$ and $ca$. The vertices of the sides of $\TT $ (as tropical lines)  are  $-(a\otimes b),-(b\otimes c)$ and $-(c\otimes a)$, again by the \emph{tropical version of Cramer's rule}. The properties of the triangle $\TT $ depend on the configuration of the six  points
\begin{equation}
a,b,c,-(a\otimes b),-(b\otimes c),-(c\otimes a),\label{eqn:six_points}
\end{equation} which are all different, generically.

\m Three tropical lines  $L_p,L_q,L_r$ also determine a \emph{tropical triangle}, $\TT' $, which is generic if the lines are not tropically concurrent. We can write
$$\TT' =\cospan(-p,-q,-r).$$
The stable intersections (by pairs) of the lines  $L_p,L_q,L_r$ are called \emph{vertices} of $\TT' $. These points should not be mixed up with the vertices  $-p,-q,-r$ of the lines. By the \emph{tropical version of Cramer's rule},  the coordinates of the vertices of $\TT' $ are $p\otimes q, q\otimes r$ and $r\otimes p$.  The properties of $\TT' $ depend on the configuration of the six  points
$$p\otimes q, q\otimes r, r\otimes p,-p,-q,-r,$$ which are all different, in the generic case.

\m
A \emph{tropical segment} is the tropical span of two points (see fig. \ref{fig:segmentos_2}).
Tropical triangles are, in general, infinite unions of tropical segments. Indeed,
\begin{equation}
\TT=\span(a,b,c)=\bigcup_{s\in\span(b,c)}\span(a,s).\label{eqn:barrido}
\end{equation}
Therefore, tropical triangles are, in general, connected non--pure two--dimensional  sets. The non--generic case arises when the points $a,b,c$ are tropically collinear (either being $a,b,c$ all different or not).
In addition, it is easy to check that tropical triangles are classically compact, both in $\T\P^2$ and in $\T^2$.

\m It is not true, in general, that the stable intersection of the tropical lines $ab$ and $bc$ gives  back the point $b$, and this makes \emph{tropical triangles trickier than classical triangles}. For example, take $a=[3,4,6]$, $b=[-2,0,8]$ and $c=[1,1,0]$ Then  $a\otimes b=[12,11,3], b\otimes c=[9,9,1]$ and $ab\cap_{\est}bc=[12,13,21]=[-1,0,8]\neq b$. The reader is encouraged to draw this example, in $Z=0$.

\m This anomalous situation for tropical triangles has been studied in \cite{Ansola_tri}, where the  definition of \emph{good tropical triangle} has been given. \label{dfn:good}   
Three  points $a,b,c$   define a \emph{good tropical triangle} if, by stable join, they give rise to three
tropical lines  $ab, bc, ca$ which,
stably intersected by pairs, yield the original points $a,b,c$, i.e., $$ca\cap_{\est}ab= a,\quad ab\cap_{\est}bc= b,\quad bc\cap_{\est}ca= c.$$ 

Good tropical triangles are characterized by six slack inequalities.
Indeed, write  the coordinates of (representatives of) $a,b,c$ as the columns of a matrix $A=(a_{ij})$ so that $c$ occupies the first column and  $a$ occupies the third.  Write
\begin{equation}
\TT_A =\span(A).\label{eqn:span}
\end{equation}
Then theorem 2 in \cite{Ansola_tri} tells us that $\TT_A \subseteq \T\P^2$ is a good tropical triangle if and only if
$$a_{12}-a_{22}\le a_{13}-a_{23}\le a_{11}-a_{21},$$
\begin{equation}
a_{23}-a_{33}\le a_{21}-a_{31}\le a_{22}-a_{32},\label{eqn:6}
\end{equation}
$$a_{31}-a_{11}\le a_{32}-a_{12}\le a_{33}-a_{13}.$$
In order to make drawings in $Z=0$ we consider the matrix
$$A_0=\AAAcero.$$
It is easy to check that the six inequalities (\ref{eqn:6})  imply the  following \emph{cardinal points condition in $Z=0$}: $\col(A_0,1)$  represents the most eastwards point, $\col(A_0,2)$ represents the most northwards one, and $\col(A_0,3)$ represents the most south--westwards one,  among the columns of $A_0$.\label{cond:puntos_card} 
More precisely, starting at $\col(A_0,1)$, we walk $a_{22}-a_{32}-a_{21}+a_{31}\ge0$ units northbound, then walk $a_{11}-a_{31}-a_{12}+a_{32}\ge0$ units westbound and we reach $\col(A_0,2)$. From there, we walk $a_{33}-a_{13}-a_{32}+a_{12}\ge0$  units south--westbound, then walk $a_{22}-a_{12}-a_{23}+a_{13}\ge0$ units southbound, to reach $\col(A_0,3)$. In a similar manner we get from $\col(A_0,3)$ to $\col(A_0,1)$ by walking  first eastbound, then north--eastbound. The distances are dictated by inequalities (\ref{eqn:6}).

  \begin{figure}[H]
 \centering
  \includegraphics[keepaspectratio,width=10cm]{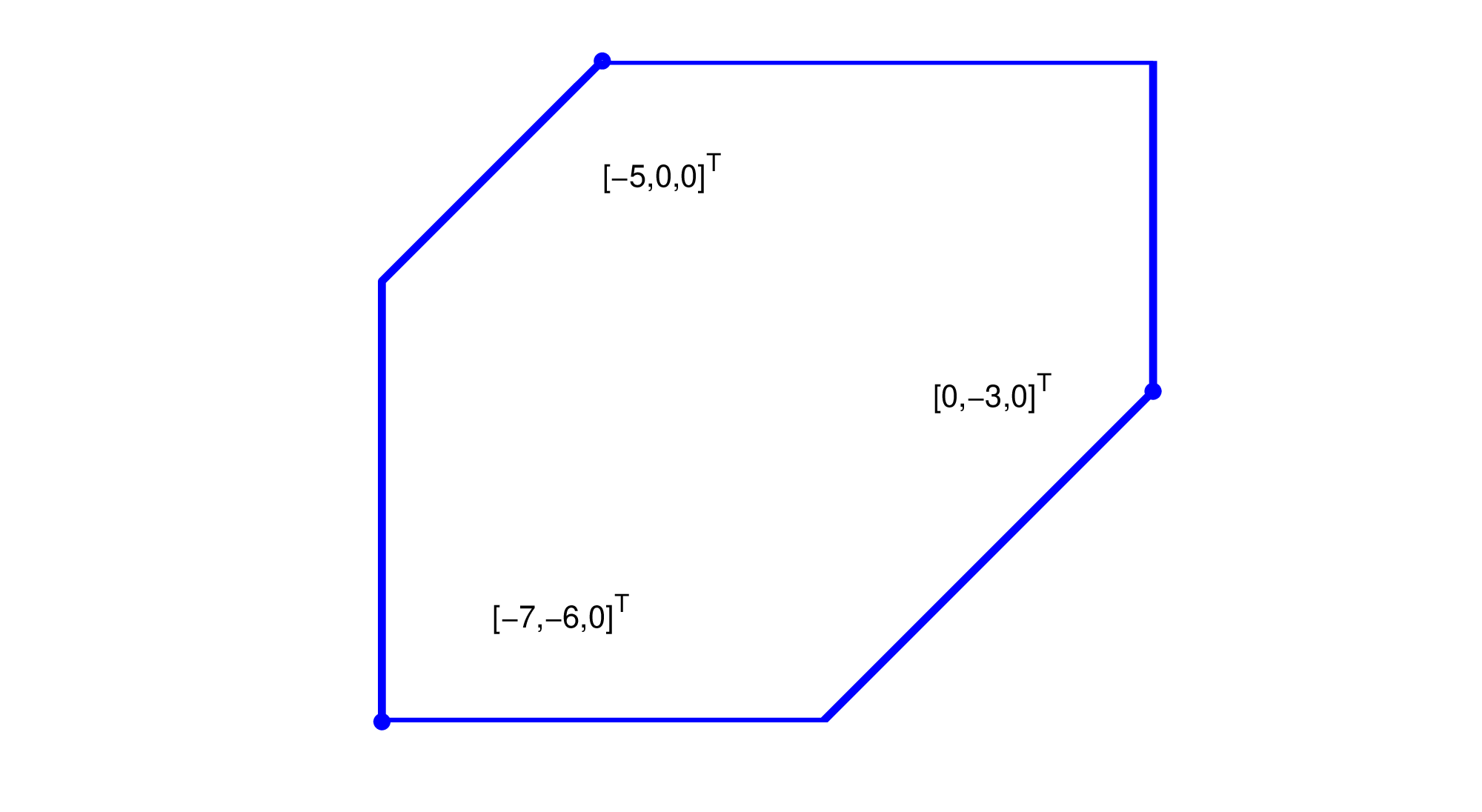}\\
  \caption{Good triangle determined by the matrix $A$.}
  \label{fig:good_5}
  \end{figure}

In figure \ref{fig:good_5} we have the good tropical triangle determined by the matrix
$$A=A_0=\left[\begin{array}{ccc}
0&-5&-7\\-3&0&-6\\0&0&0
\end{array}\right].$$
In simple words, in  $Z=0$,
good tropical triangles are nothing but  classical hexagons, pentagons, quadrangles or triangles having slopes $0,1$ and $\infty$, where the inequalities (\ref{eqn:6}) provide the integer length of the sides.
The hexagons are obtained  by chopping off two  corners, in a classical rectangle of sides parallel to axis $X,Y$, see figure \ref{fig:rectangle_13}. The pentagons, quadrangles or triangles arise from one such hexagon, when one or more sides collapse to a point. In figure \ref{fig:goods_6} we see a few good triangles.
Here is one more way to describe good triangles: any  $a\le b$, $c\le d$ and $e\le f$ in $\R$ define the following good tropical triangle in $Z=0$:
\begin{equation}\label{eqn:good1}
\TT=\{(x,y)\in \R^2: a\le x\le  b,\ c\le y\le d,\ e\le y-x\le  f\}.
\end{equation}

\begin{figure}[H]
 \centering
  \includegraphics[keepaspectratio,width=10cm]{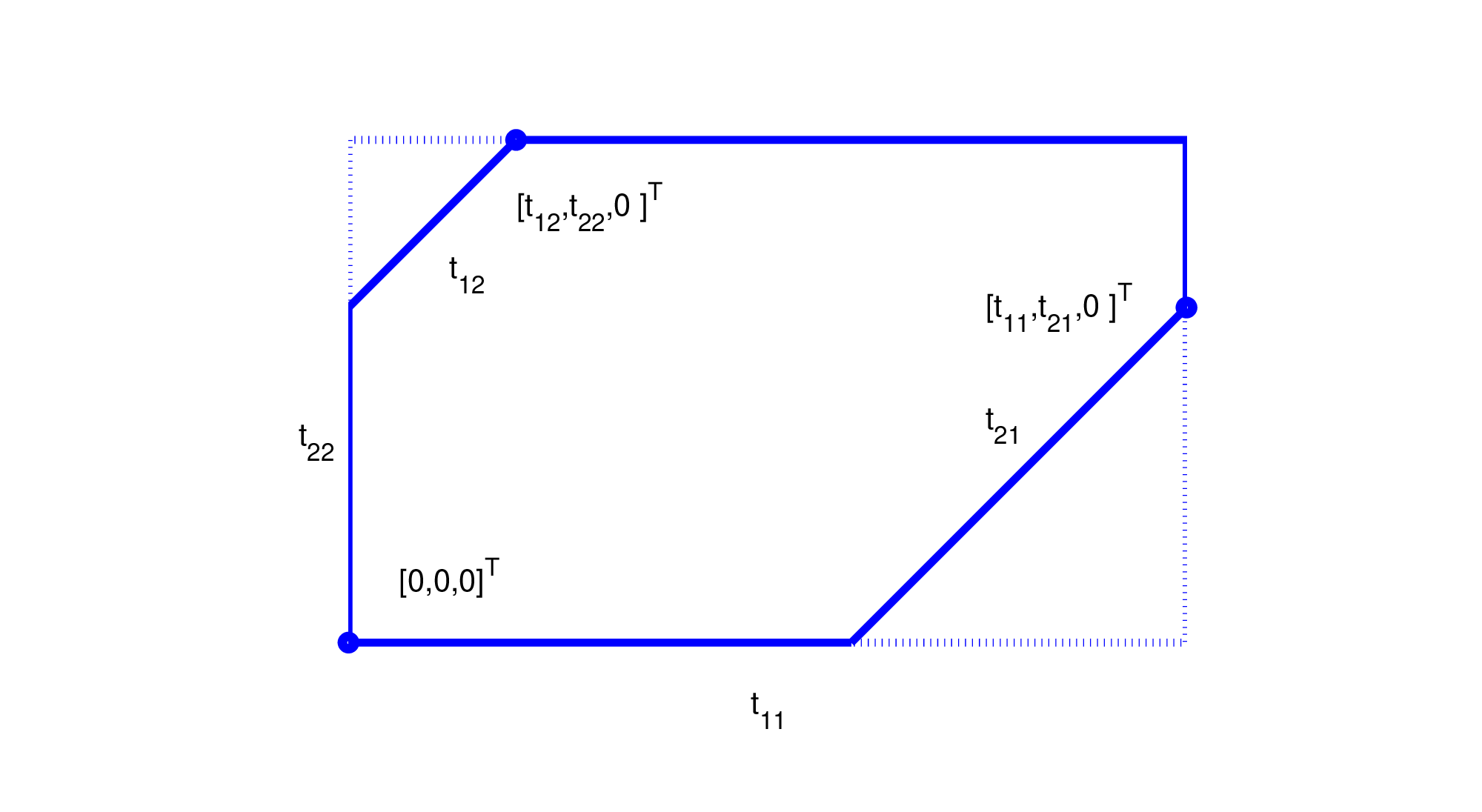}\\
  \caption{A good tropical triangle is a classical rectangle with two corners chopped off.}
  \label{fig:rectangle_13}
  \end{figure}

   \begin{figure}[H]
 \centering
  \includegraphics[keepaspectratio,width=8cm]{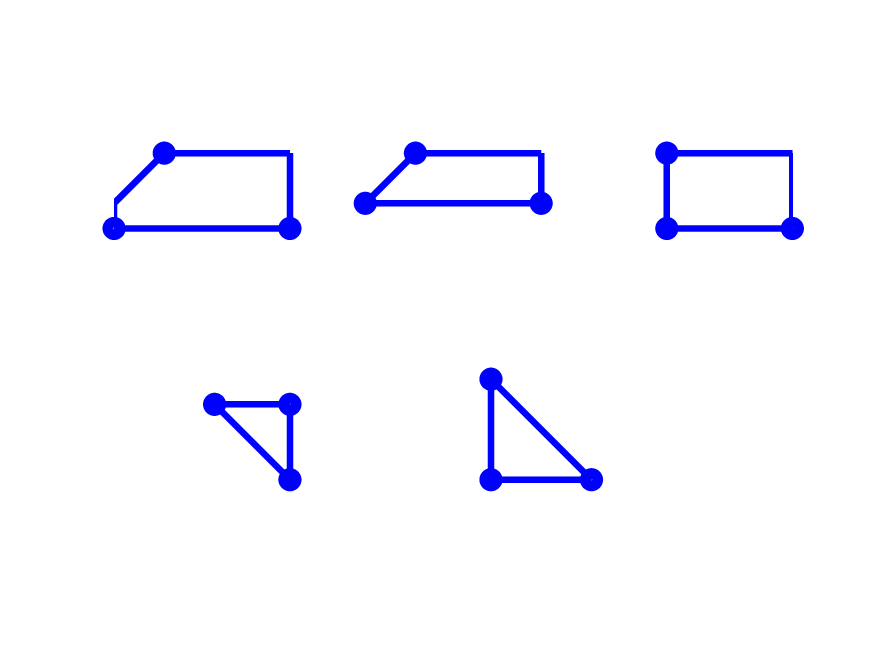}\\
  \caption{Some good tropical triangles.}
  \label{fig:goods_6}
  \end{figure}

In the classical plane $\R^2$ we have  the following norm \label{dfn:norm} 
$$\|p\|=\max\{\|p_1\|, \|p_2\|, \|p_1-p_2\|\}, \qquad  p\in\R^2.$$
It is easy to check that $\|p\|$ is the integer length  of the tropical segment $\span(p,0)$, if we identify $p=(p_1,p_2)$ with $[p_1,p_2,0]$. For instance, $\|(-5,-2)\|=5$, $\|(-3,5)\|=8$. The unit ball and some radii in it are shown in figure \ref{fig:unit_ball_7}. Given real  points $p,q\in\R^2$ the \emph{tropical distance} between $p$ and $q$ is $\|p-q\|$,  by definition. It is the integer length  of the tropical segment $\span(p,q)$.
This  is connected with the Hilbert projective metric appearing in
\cite{Cohen,Develin,Joswig_S} and  to the  range seminorm of \cite{Cuninghame_B}.

  \begin{figure}[H]
 \centering
  \includegraphics[keepaspectratio,width=10cm]{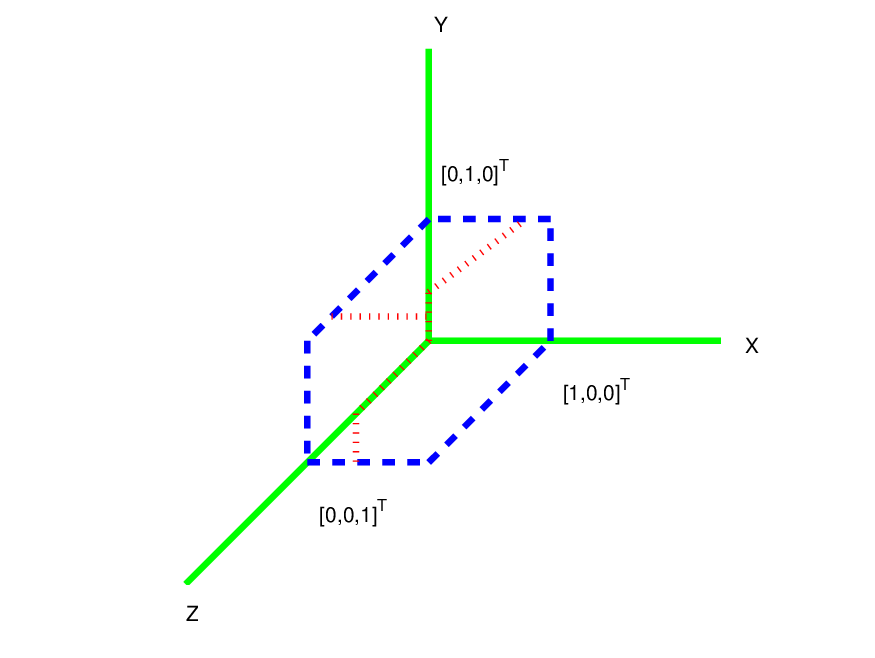}\\
  \caption{Axes and unit ball in the tropical the plane. Some rays are shown in dotted lines.}
  \label{fig:unit_ball_7}
  \end{figure}

\m In $\T\P^2$, let $V$ be the tropical span of a finite family of points.  In \cite{Cohen,Develin,Joswig_S}, a  \label{dfn:projector} \emph{projector map} (or \emph{nearest point map}) $\rho:\T\P^2\to V$ is considered. It satisfies $\rho\circ\rho=\rho$  and $\rho|_V=\id_V$. For a point $p\in\T\P^2\setminus V$, the image $\rho(p)$ is computed as follows: fix a representative $p'\in\T^3$ of $p$  and, for each generator $v$ of $V$, choose a representative $v'\in\T^3$, optimal for the condition $v'\le p'$ (meaning $v'_j\le p'_j$, for $j=1,2,3$ and equality is attained for, at least, one $j$). Tropically add all such $v's$ and then, take $\rho(p)$ to be the point in $\T\P^2$ represented by the sum.  In \cite{Cohen,Develin,Joswig_S} it is shown that $\rho(p)\in V$ minimizes the tropical distance $\|p-q\|$, when $q$ runs through $V$. In general, there are  infinitely many  points  $q$ in $V$ minimizing such a distance, in addition to $\rho(p)$. Indeed,  consider tropical balls $B(d)$ centered at $p$ of increasing radius $d$ and take the minimum
$d>0$ such that the intersection $B(d)\cap V$ is non--empty. Then $B(d)\cap V$ is the set of minimizing points.

\section{Matrices, maps and pictures in $Z=0$}
All arrays will have entries in $\T$. Arrays will be denoted by capital letters $A,B,C$, $N,P,Q$, etc.
Tropical matrix addition and multiplication are defined in the usual way, but using the  tropical operations $\oplus$ and $\odot$, instead of the classical ones.
 Any array all whose entries are zero will be denoted by 0. Given two arrays of the same size $A=(a_{ij})$, $B=(b_{ij})$, we will write $A\le B$ if $a_{ij}\le b_{ij}$, for all $i,j$.

\m
We will  deal with  $3\times 3$ matrices. The \emph{tropical determinant} of a $3\times 3$
 matrix $A=(a_{ij})$ (also called \emph{tropical permanent})
 is defined as $$|A|_{trop}=\max_{\sigma\in \Sigma_3}\{a_{1\sigma(1)}+a_{2\sigma(2)}+a_{3\sigma(3)}\},$$
where $\Sigma_3$ denotes the symmetric group in $3$ symbols. A
matrix is  \emph{tropically singular} if \emph{the maximum
in the tropical determinant is attained, at least, twice}.
Otherwise the matrix is \emph{tropically regular}, or it is said to have a \emph{strong permanent}. These are all standard definitions.

\m Given a matrix $A$, the $j$--th column (resp. row) of $A$ will be denoted $\col(A,j)$ (resp. $\row(A,j)$). The triple of diagonal entries of $A$   will be denoted $\diag(A)$. Moreover, if $t\in\R^3$, then $\diag(t)$ will denote the matrix whose diagonal is $t$, the rest of entries being equal to $-\infty$; such matrices will be called  \emph{diagonal matrices}. A \emph{permutation matrix} \label{dfn:permutation_matrix} is a matrix  obtained from a diagonal matrix, by permuting some of its rows or permuting some of its columns.
A particular case is the \emph{tropical identity matrix},
$I=\diag(0)$. Another example is \label{matrix:perm} $$P_{12}=\left[ \begin{array}{ccc} -\infty&0&-\infty\\ 0&-\infty&-\infty\\ -\infty&-\infty&0\\ \end{array}\right].$$ Any permutation matrix $P$ has a tropical inverse $P^{\odot-1}$, meaning $P\odot P^{\odot-1}=P^{\odot-1}\odot P=I$.

\m From now on, \emph{points in $\T\P^2$ will be denoted by columns}, for convenience. We often \emph{identify a  $3\times 3$ matrix $A$ with the three points in $\T\P^2$ represented by its columns}.

\m The reader can easily check that left--multiplication by the matrix $P_{12}$  exchanges coordinates $X$ and $Y$:
$$P_{12}\odot \co xyz=\co yxz.$$
A triple $t=(t_1,t_2,t_3)\in \R^3$ gives rise to a \emph{translation}  in $\T\P^2$:
$$\co XYZ\mapsto  \co {X'}{Y'}{Z'}=\co{t_1+X}{t_2+Y}{t_3+Z}=\diag(t)\odot \co XYZ,$$
By a \emph{change of projective coordinates} in the tropical projective plane $\T\P^2$ we mean
left--multiplying coordinates by a permutation matrix.
Therefore, a change of projective coordinates amounts to the composition 
of a translation and a permutation of coordinates.  
Notice that right--multiplying  $A$  by a diagonal matrix does not change the columns of $A$ in $\T\P^2$; it only changes the representatives of them.

\m
 All  pictures will be done in the affine tropical plane $Z=0$. \label{dfn:work_in Z=0} In order to do so, from a given matrix $A$ we compute
 the matrix \begin{equation}
A_0=A\odot\diag(-\row(A,3)).\label{eqn:A_0}
\end{equation}

\m From now on, \emph{suppose that $A$ is real}.  Our aim is  to describe  the map $f_A:\T\P^2 \to \T\P^2$
$$\co xyz\mapsto A\odot\co xyz=\left[
\begin{array}{c}
\max\{a_{11}+x,a_{12}+y,a_{13}+z\}\\
\max\{a_{21}+x,a_{22}+y,a_{23}+z\}\\
\max\{a_{31}+x,a_{32}+y,a_{33}+z\}\\
\end{array}
\right].$$  First, notice that proportional matrices   $A$ and $\lambda\odot A$ determine the same map $f_{A}=f_{\lambda\odot A}$, any $\lambda\in\R$.
The simplest examples of maps $f_A$ arise for $A=I$ (resp.  $A=0$),  the map being the identity (resp. constant). It is constant also for $f_{A\odot 0}$ and  $f_{0\odot A}$, because all the columns  of $A\odot 0$ (resp. $0\odot A$)  represent the same point in $\T\P^2$. \label{asser:constant}

\m The map $f_A$ is obviously continuous and piecewise--linear. The image $\im f_A$ is  the tropical triangle spanned by $A$, meaning that it is spanned by  the columns of $A$:
\begin{equation}\label{eqn:im}
\TT_A=\im f_A=\span(A).
\end{equation}
\emph{The map $f_A$ is not surjective}, since no finite family of points with finite coordinates  span the whole $\T\P^2$; this  is well--known (see, e.g.,  \cite{Wagneur}).
Moreover,
if $r,s\in\R$  are negative and big enough, we have
$$A\odot \co0{r}{s}=\col(A,1),\ A\odot \co{s}0{r}=\col(A,2),\ A\odot \co{r}{s}0=\col(A,3).$$
Therefore, $f_A$ is \emph{locally constant} on  three big chunks of $\T\P^2$, called \emph{corners}; see figure \ref{fig:montr_corners_1}.\label{dfn:corners} In particular, \emph{$f_A$ is not injective.} \label{asser:injective}

\begin{figure}[H]
 \centering
  \includegraphics[width=6cm,keepaspectratio]{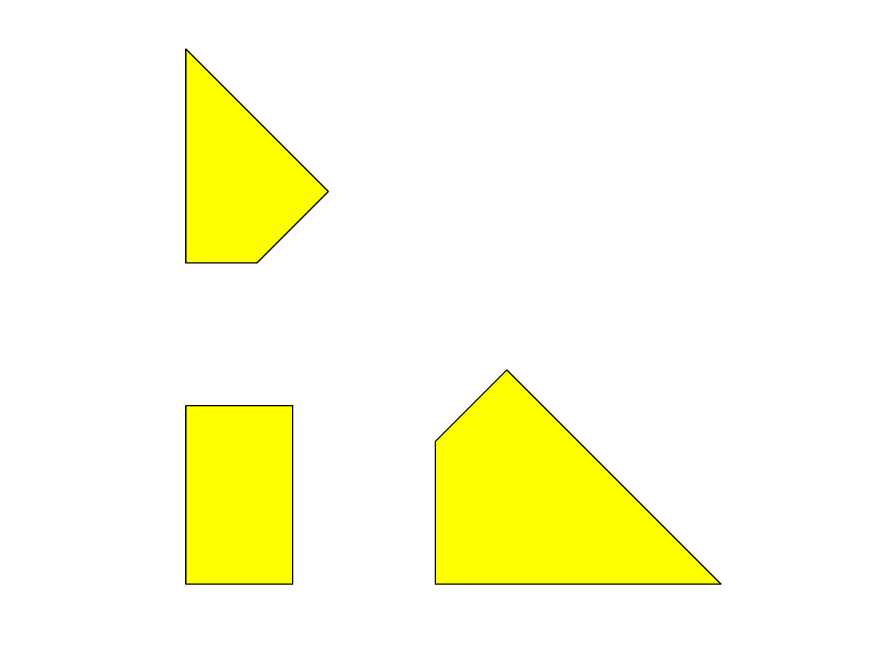}\\
  \caption{Corners.}
  \label{fig:montr_corners_1}
  \end{figure}

\m Let us see how do these  corners arise. First, the matrix $A$ defines
three tropical lines  $\su A$, because the $j$--th row of $A$ provides a tropical linear form
$$a_{j1}\odot X\oplus a_{j2}\odot Y\oplus a_{j3}\odot Z=\max\{a_{j1}+ X, a_{j2}+ Y, a_{j3}+Z\}.$$
 The vertices of $\su A$ are (represented by) the rows of $-A$, i.e., the columns of $-A^T$. Thus we have another tropical triangle here, namely \begin{equation}
\TT^A=\cospan(-A^T).\label{eqn:cospan}
\end{equation}
 The lines $\su A$ (or, rather, the matrix $A$) induce a
 \emph{cell decomposition on $\R^2$}, \label{dfn:cell_decom} denoted $\CC^A_\R$ (see   \cite{Develin} for an isomorphic cell decomposition).
The decomposition $\CC^A_\R$ consists of, at most,  31 cells, and this is the generic case. Every  cell $\Gamma$ is relatively open, i.e.,  $\Gamma$ is open inside its affine hull in $\R^2$.

In $\CC^A_\R$ we have:
\begin{itemize}
\item ten two--dimensional cells:    one bounded cell, denoted $B^A$, the three already mentioned \emph{corners} (denoted $\su {C^A}$), six unbounded cells (parallel to some tropical coordinate axis $X,Y$ or $Z$), 
\item fifteen  one--dimensional cells: nine unbounded cells (parallel to some coordinate axis) and six bounded cells,
\item six  zero--dimensional cells or points.
\end{itemize}
Notice that the union of all the bounded cells above is nothing but $\TT^A$. Moreover, $\overline{B^A}$ is the union of some bounded cells.

For later use, bounded cells will also be called \emph{central cells}; all other cells will be called \emph{peripheral cells}.
In figure \ref{fig:cell_decom_good_8} we find the 31 cells described above,  and  figure \ref{fig:cell_decom_bad_9} represents the cell decomposition induced by the matrix
\begin{equation}
A=\left[\begin{array}{ccc}
0&-1&-5\\-4&0&-2\\-1&-4&0
\end{array}\right], \ (-A^T)_0=\left[\begin{array}{ccc}
-5&2&1\\-4&-2&4\\0&0&0
\end{array}\right].\label{eqn:cell_decom_bad_9}
\end{equation}
\begin{figure}[H]
\centering
  \includegraphics[keepaspectratio,width=10cm]{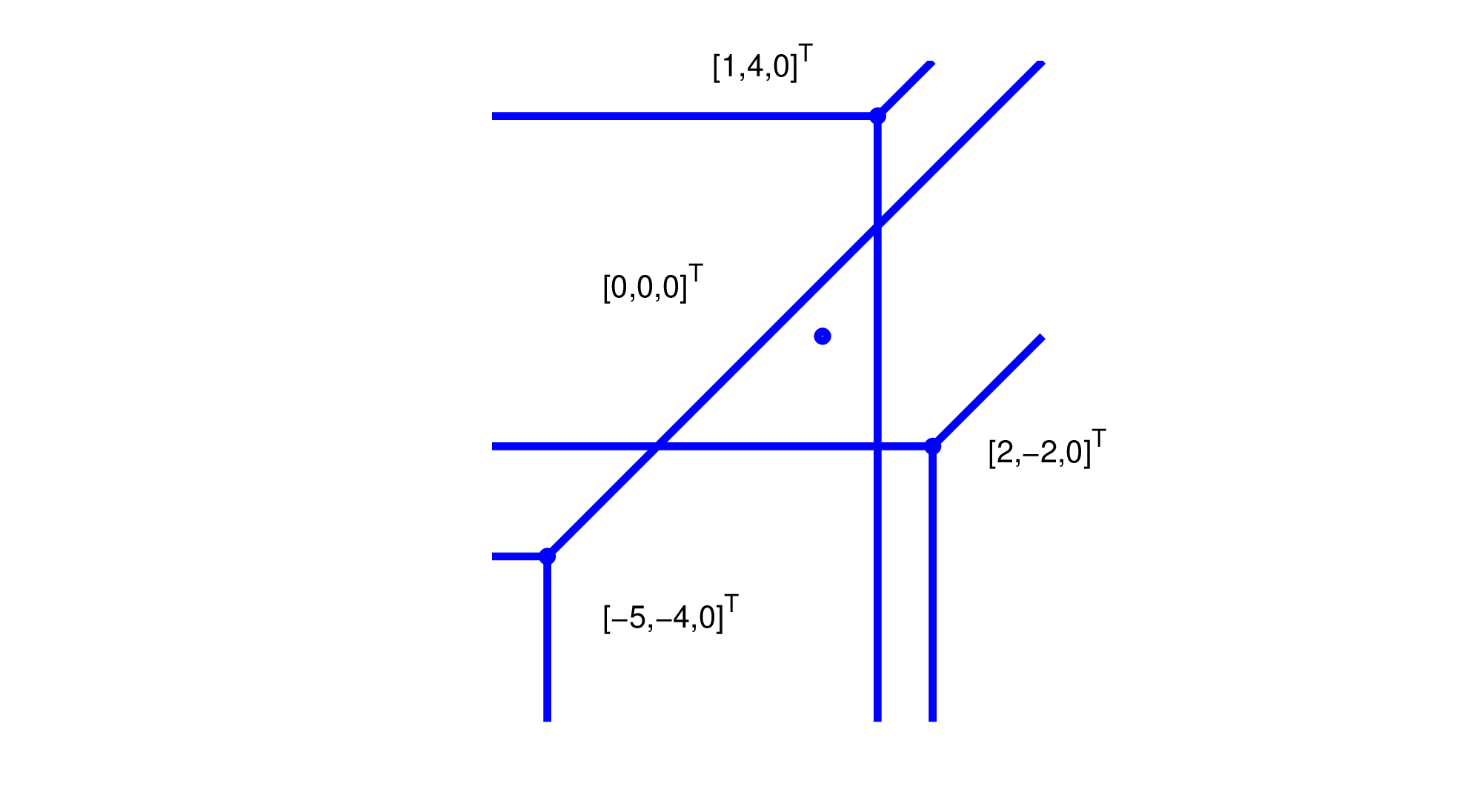}\\
  \caption{Cell decomposition $\CC^A_\R$ induced by matrix $A$  in (\ref{eqn:cell_decom_bad_9}).}
  \label{fig:cell_decom_bad_9}
  \end{figure}

  \begin{figure}[H]
 \centering
  \includegraphics[keepaspectratio,width=10cm]{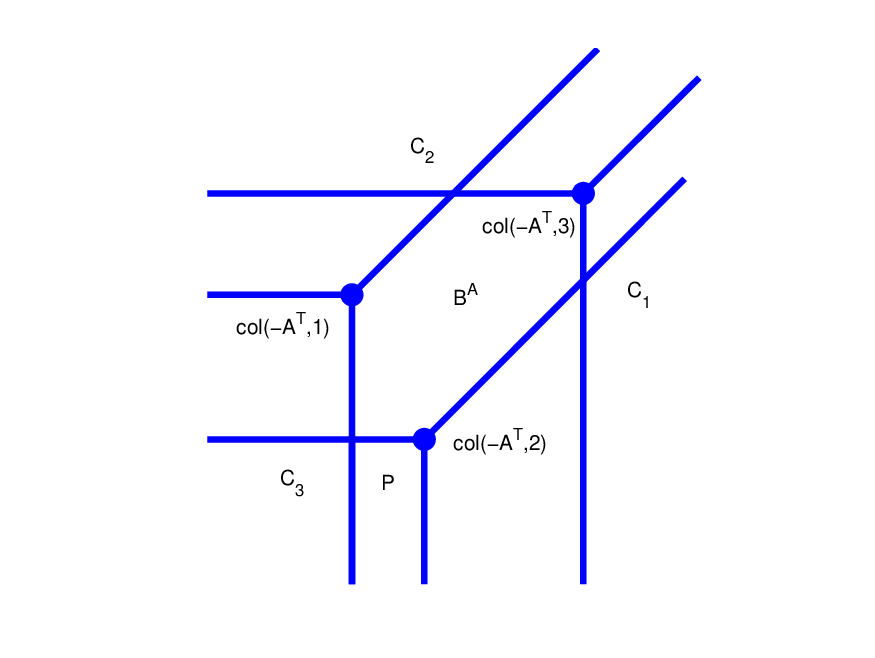}\\
  \caption{The 31 cells  in the cell decomposition $\CC^A_\R$ induced by some matrix $A$.}
  \label{fig:cell_decom_good_8}
  \end{figure}

$\T\P^2$ is a compactification of $\R^2$; \label{asser:compact2} cf.  p. \pageref{asser:compact}.  The set of \emph{boundary points} of this compactification is
 $$\partial=\T\P^2\setminus \R^2=\{[x,y,z]^T: x=-\infty \text{\ or\ }y=-\infty \text{\ or\ }z=-\infty.\}$$
Therefore, the cell decomposition $\CC^A_\R$ induces a cell decomposition $\CC^A$ \label{dfn:cell_decom2} of $\T\P^2$, which, in addition to all the cells in $\CC^A_\R$, contains
  \begin{itemize}
  \item $k$ one--dimensional cells,
  \item $k$ zero--dimensional cells or points,
  \end{itemize}
  for some $3\le k\le 12$. The union of these additional cells is $\partial$. Notice that in most  of our figures, we have not drawn $\partial$.
Since $\T\P^2$ is compact, it no longer makes  sense  talking about unbounded cells, but we have already introduced the alternative term \emph{peripheral}. Recall  that $B^A$ is \emph{ the central two--dimensional cell}. For instance, $B^A$ is empty, if $A=0$.

\m The description of the  map $f_A:\T\P^2\to\T\P^2$ is particularly easy  when the tropical triangles $\TT_A$ and $\TT^A$  fit as much as possible: then the action of $f_A$ is easily described on the closure of each cell $\Gamma$ of
 the  decomposition  $\CC^A$; see  theorem \ref{thm:f_A}.\label{dfn:fit}

\section{Normal matrices}

By definition, a  matrix $A$ is \emph{normal} if $\diag(A)=0$ and $A\le0$; in symbols,
\begin{equation}
I\le A\le0\label{eqn:normal}
\end{equation} see \cite{Butkovic_S}; in \cite{Cuninghame_New} a matrix $A$ such that $I\le A$ is called \emph{increasing}.
For any matrix $A$ there exist permutation matrices $P,Q$  such that the product
\begin{equation}\label{eqn:normalizacion}
N=P\odot A\odot Q
\end{equation}
is normal. The matrix $N$ is called \emph{a normalization of $A$}. The  \emph{Hungarian method} (see \cite{Butkovic_S,Kuhn,Papa}) is an algorithm to obtain such $N,P,Q$. A matrix $A$ admits several normalizations.
Notice that  the columns of $A$ and the columns of $A\odot Q$ represent the same points in $\T\P^2$, given perhaps in a different order.
And the columns of $N$ are a just a translation of those points.

\m As in classical mathematics, the product of matrices corresponds to the composition of  maps:  $$f_N=f_P\circ f_A\circ f_Q.$$
Now,  $f_P$ and $f_Q$ are changes of projective coordinates, so that  \emph{in order to study the map $f_A$, we may assume that $A$ is normal, up to  changes of coordinates.}

\m  A normal matrix $A$ satisfies $I\le A\le0$, and therefore
\begin{equation}
I\le A\le A^{\odot2}\le A^{\odot3}\le\cdots\label{eqn:desigualdades}
 \end{equation} and,  for any natural number $m$,
 \begin{equation}
A^{\odot m+1}\le A^{\odot m}\odot 0\label{eqn:desigualdades_m}
 \end{equation}
since tropical multiplication by any matrix is monotonic (because max and $+$ are monotonic). And  the map  $f_{A^{\odot m}\odot 0}$  is  constant, as explained in p. \pageref{asser:constant}.

In corollary \ref{cor:sqrt} we will  see that the tropical powers of $A$ are simpler than $A$ (in the sense that they depend on fewer parameters), when $A$ belongs to a particular class of normal matrices. This  simplification  will carry over to  the corresponding maps
$$\id=f_I,f_A,f_{A^{\odot2}},f_{A^{\odot3}},\ldots, \const.$$

\m Consider  the cell decomposition $\CC^0$ induced by  the zero matrix on $\T\P^2$; it is just the cell decomposition given by the tropical line $L_0$. It has three two--dimensional cells (corners),  which have the following description in $Z=0$:  $$C^0_1=\{0< x, y< x\},\qquad C^0_2=\{0< y, x< y\},\qquad C^0_3=\{x< 0, y< 0\}.$$
The  \emph{geometric meaning of normality} is the following:  if    $A$ is a $3\times 3$   normal  matrix then,
\begin{equation}
 \col(A_0,j)\in \overline{C^0_j},\quad\text{for all}\quad j=1,2,3. \label{eqn:col_0}
\end{equation}

\m Next  we  define several operators on  matrices and then we study the relationship among them. Of course, we are particularly interested in these operators acting on normal matrices.

\m  For any $k\in\N$, the tropical $k$--th power of $A$,  denoted  $A^{\odot k}$,  takes  normal matrices to normal matrices. The transpose $A^T$ of a normal matrix $A$ is a normal matrix.
These operators commute with each other.  Warning: $(-A)^{\odot2}\neq A^{\odot2}$, in general. Also, $(A_0)^{\odot2}\neq (A^{\odot2})_0$, in general.

\m We  introduce the \emph{tropical adjoint}  of $A$, denoted $\widehat{A}$. By definition, $\widehat{A}=(\alpha_{ij})$, where $\alpha_{ij}$ is the tropical cofactor of $a_{ji}$. In other words,
\begin{equation}
\row(\widehat{A},j)=\col(A,j-1)\otimes\col(A,j+1), \label{eqn:adjoint}
\end{equation}
for $j=1,2,3, \mod 3$.
Last, we define an auxiliary  matrix operator,  $\breve{A}=(\beta_{ij})$,  by the formulas
$$\beta_{ii}=0, \qquad \beta_{ij}=a_{ik}+a_{kj},\qquad \text{\ if \ } i\neq j \text{\ and\ } \{i,j,k\}=\{1,2,3\}.$$

\begin{lem}\label{lem:normal}
If $A$ is $3\times 3$  normal, then
\begin{enumerate}
\item $\breve{A}$ is normal and $\widehat{A}=A\oplus \breve{A}=A^{\odot 2}$,\label{cond:primera}
\item  $\widehat{A}$ is normal,\label{cond:segunda}
\item  $A^{\odot 2}=A^{\odot 3}$, \label{cond:tercera}
\item  every point in $\TT_{A^{\odot2}}=\span(A^{\odot2})$ is fixed by $f_A$,\label{cond:cuarta}
\item  zero (the neutral element for tropical multiplication) is an eigenvalue of $A$. \label{cond:quinta}
\end{enumerate}
\end{lem}
\begin{proof}  A straightforward computation yields (\ref{cond:primera}) and then (\ref{cond:segunda}) follows.
Now, multiplication by $A$ is a monotonic operator; so that the equality in (\ref{cond:primera}) implies $A^{\odot 3}=\max\{A^{\odot 2}, A\odot \breve{A}\}$. Now, a simple computation shows that  $A\odot \breve{A}=A\oplus \breve{A}$, whence $A^{\odot 3}=A^{\odot 2}$ follows. Finally, (\ref{cond:cuarta}) follows from (\ref{cond:tercera}) and (\ref{cond:quinta}) follows from (\ref{cond:cuarta}).
\end{proof}

Lemma \ref{lem:normal}  follows from \cite{Yoeli},  where real matrices of any size $n$  are considered.
The so called \emph{Kleene star of $A$} (or \emph{strong closure of $A$}) is defined as
$$A^*=I\oplus A\oplus A^{\odot2}\oplus A^{\odot3}\oplus\cdots,$$ if the limit exists, see\cite{Akian_HB, Carre}. If $A$ is a $3\times 3$ normal matrix, then $A^*=A^{\odot2}$, but we will not use this.

\begin{lem}\label{lem:A=A2}
For a $3\times 3$ normal matrix $A$, the following  are equivalent:
\begin{enumerate}
\item $\breve{A}\le A$,
\item $A=A^{\odot 2}$, i.e., $A$ is idempotent, \label{cond:square}
\item  $\TT_A $ is  good.\label{cond:good}
\end{enumerate}
\end{lem}

\begin{proof}
The equivalence follows from lemma \ref{lem:normal} and the six inequalities (\ref{eqn:6}), letting $a_{jj}=0$, for $j=1,2,3$. Indeed, we obtain
$$a_{23}+a_{31}\le  a_{21},\qquad a_{32}+a_{21}\le  a_{31},$$
\begin{equation}
a_{13}+a_{32}\le  a_{12},\qquad a_{31}+a_{12}\le  a_{32},\label{eqn:good}
\end{equation}
$$a_{12}+a_{23}\le  a_{13},\qquad a_{21}+a_{13}\le  a_{23}.$$
\end{proof}

Suppose $A=(a_{ij})$ is normal and consider
$$A_0=\left[\begin{array}{ccc}
-a_{31}&a_{12}-a_{32}&a_{13}\\
a_{21}-a_{31}&-a_{32}&a_{23}\\0&0&0
\end{array}\right].$$
By a translation, we can assume that  $a_{13}=a_{23}=0$, so
that $\col(A_0,3)=0$. Write
\begin{equation}
t_{11}=-a_{31}, \ t_{22}=-a_{32}, \ t_{21}=a_{21}-a_{31}, \
t_{12}=a_{12}-a_{32},
\end{equation} so that

    \begin{equation}
A_0=\left[\begin{array}{ccc}t_{11}&t_{12}&0\\
t_{21}&t_{22}&0\\
0&0&0\\
\end{array}
\right]
\end{equation}
If,  in addition,  $A$ is idempotent, then
\begin{equation}
0\le t_{11},t_{22},\quad 0\le t_{21},t_{12}\le
\min\{t_{11},t_{22}\}\label{eqn:good2}
\end{equation}
and these $t_{ij}$ provide a \emph{parameter space for good
tropical triangles, up to translation}; see figure
\ref{fig:rectangle_13}.

The six points listed in
(\ref{eqn:six_points}) are (represented by) the columns of
$A$ and of $-\widehat{A}^T$, according to the definition of
adjoint matrix. They determine the shape of the tropical triangle
$\TT_A$.

\begin{lem}\label{lem:new}
If $A$ is a $3\times 3$ idempotent normal matrix, then
$A^T=(A^{\odot 2})^T=\widehat{A}^T$. In particular, $\TT_A$  is determined by
 the columns of $A$ and of $-A^T$.\qed
\end{lem}
Figure \ref{fig:normal_idempot_11}  illustrates the former lemma.

\begin{figure}[H]
 \centering
  \includegraphics[keepaspectratio,width=10cm]{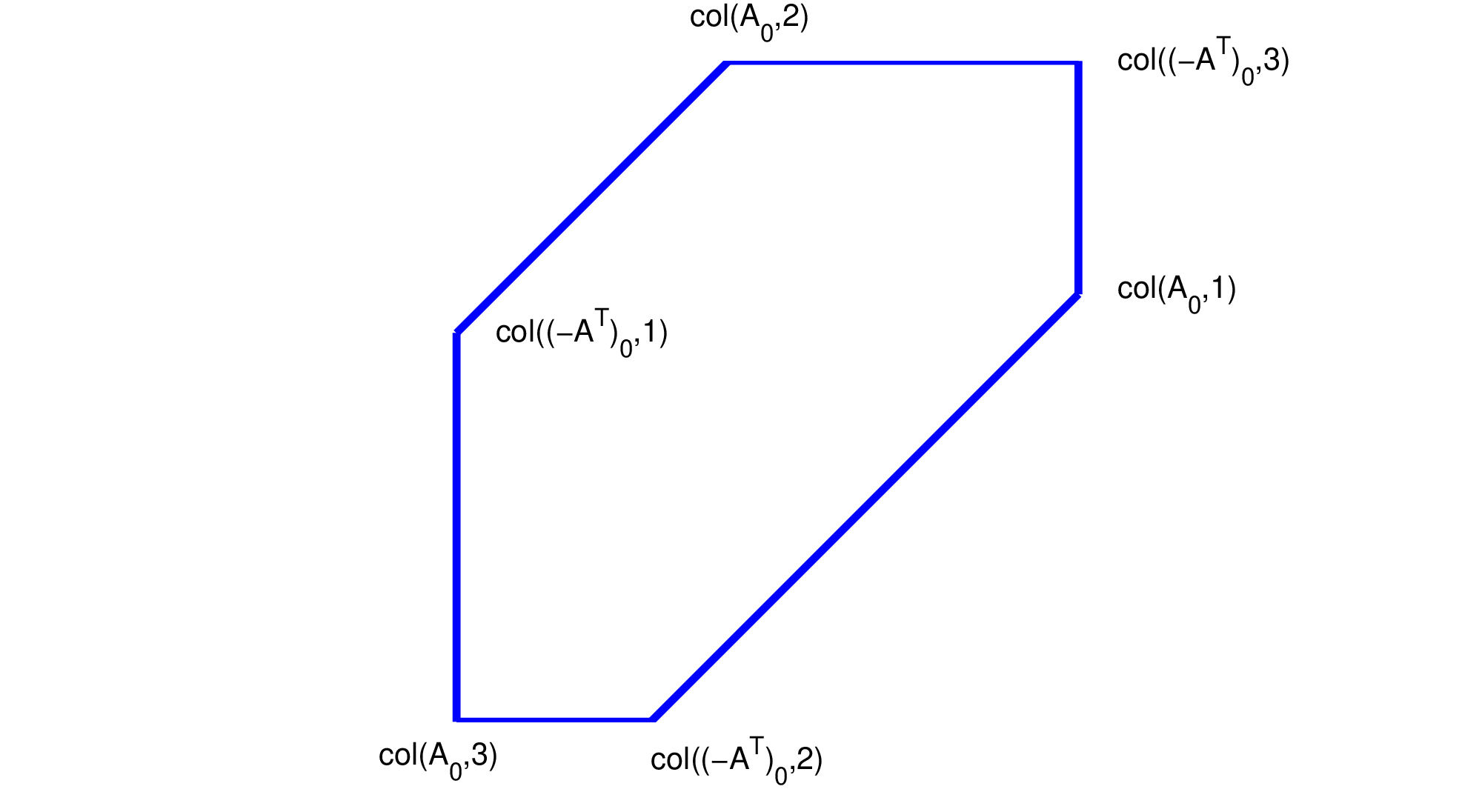}\\
  \caption{Tropical triangle associated to an idempotent normal matrix.}
  \label{fig:normal_idempot_11}
  \end{figure}

\section{Canonical normalization}
The geometric meaning of canonical normalization is  \emph{getting
pictures centered  at the origin of $Z=0$}.  We have
two equivalent ways to achieve this goal: \emph{upper and lower
canonical normalization}. The difference is irrelevant: just
an exchange of coordinates $X$ and $Y$. Our choice will be \emph{lower canonical
normalization}. We have used $L$ (resp. $U$) to mean \emph{lower}  (resp. \emph{upper}).

For each  $d,\su d\in\R$ consider the matrix
\begin{equation}
L(d,\su d)=\LLL\label{eqn:LLL}
\end{equation}
Notice the symmetric role played by $d_1$  with respect to
$X$, $d_2$  with respect to $Y$ and $d_3$  with respect to
$Z$.
We will use the matrices
\begin{equation}
L(d,\su d)_0=\LLLcero,\label{eqn:LLLcero}
\end{equation}
\begin{equation}
(-L(d,\su d)^T)_0=\LLLTmenoscero.\label{eqn:LLLTmenoscero}
\end{equation}

It is easy to check that if  $d\ge0$ and $-d\le d_j$, for  $j=1,2,3$ then $L(d,\su d)$ is normal. If, in addition, $\su d\ge0$, then $A$ is also idempotent. In this case, $\TT_A$ reduces to a segment if and only if $d=d_j=d_{j+1}=0$, for some $j=1,2,3$ modulo 3,\label{rem:l2} and $\TT_A$ reduces to a point if and only if $d=d_1=d_2=d_3=0$.\label{rem:l3}

\begin{figure}[H]
 \centering
  \includegraphics[keepaspectratio,width=12cm]{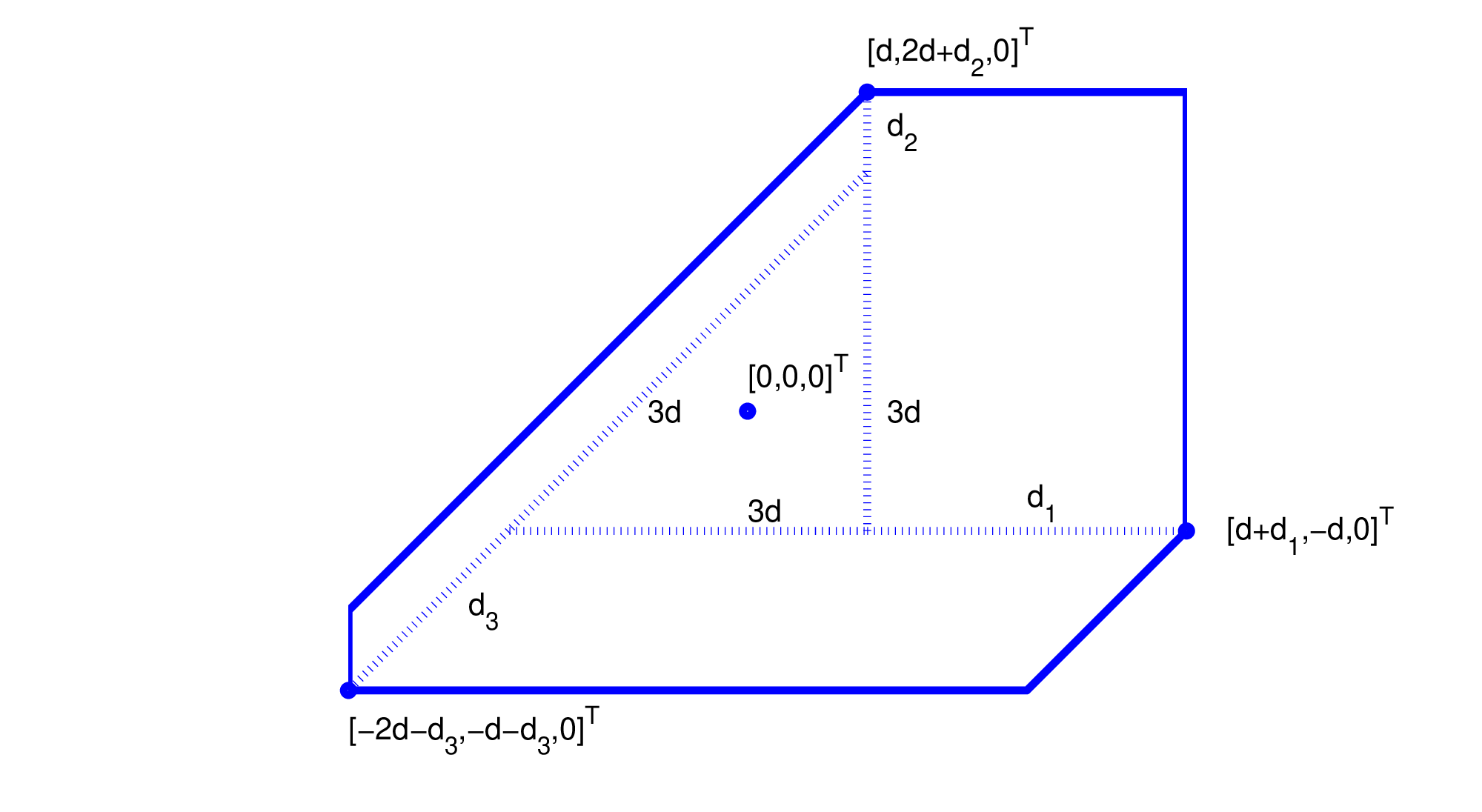}\\
  \caption{Tropical triangle given by the matrix $L(d,\su d)$, in $Z=0$. Dotted lines are auxiliary.}
  \label{fig:good_L_10}
  \end{figure}

In a similar fashion we can consider  the matrix $$U(d,\su
d)=\UUU.$$  Notice that   $P_{12}U(d,\su d)P_{12}=L(d,d_2,d_1,d_3)$. 

\begin{lem}[Lower canonical normalization for an idempotent normal matrix]\label{lem:canonicalA=A2}
If $A$ is a $3\times 3$ idempotent normal matrix, then there exist unique $d, \su d\ge0$ and
there exist permutation matrices $P,Q$ such that
$L(d,\su d)=P\odot A\odot Q$.
\end{lem}
\begin{proof} A translation allows us to assume that $\col(A_0,3)=0$.   Then
$$A_0=\left[\begin{array}{ccc} t_{11}&t_{12}&0\\t_{21}&t_{22}&0\\0&0&0
\end{array}\right],\qquad A=\left[\begin{array}{ccc} 0&t_{12}-t_{22}&0\\t_{21}-t_{11}&0&0\\-t_{11}&-t_{22}&0
\end{array}\right],$$ with
$0\le t_{11},t_{22}$ and $0\le t_{21},t_{12}\le
\min\{t_{11},t_{22}\}$.

Assume that $t_{21}\le t_{12}$ (see figure \ref{fig:rectangle_19}). Then we take   $d=\frac{1}{3}(t_{12}-t_{21})$, $d_1=t_{11}-t_{12}$, $d_2=t_{22}-t_{12}$, $d_3=t_{21}$ and $Q=\diag(d_3+2d,d_3+d,0)$, obtaining $L(d,\su d)=Q^{\odot-1}\odot A\odot Q$.

Now,  assume that $t_{21}\ge t_{12}$ (see figure \ref{fig:rectangle_18}). Then we take   $d=\frac{1}{3}(t_{21}-t_{12})$, $d_2=t_{11}-t_{21}$, $d_1=t_{22}-t_{21}$, $d_3=t_{12}$ and $R=\diag(d_3+d,d_3+2d,0)$, obtaining $U(d,d_2,d_1,d_3)=R^{\odot-1}\odot A\odot R$. Now $L(d,\su d)=P_{12}\odot U(d,d_2,d_1,d_3)\odot P_{12}$.


The uniqueness of $d,\su d$ follows from the geometric meaning of these parameters.
\end{proof}

   \begin{figure}[H]
 \centering
  \includegraphics[keepaspectratio,width=10cm]{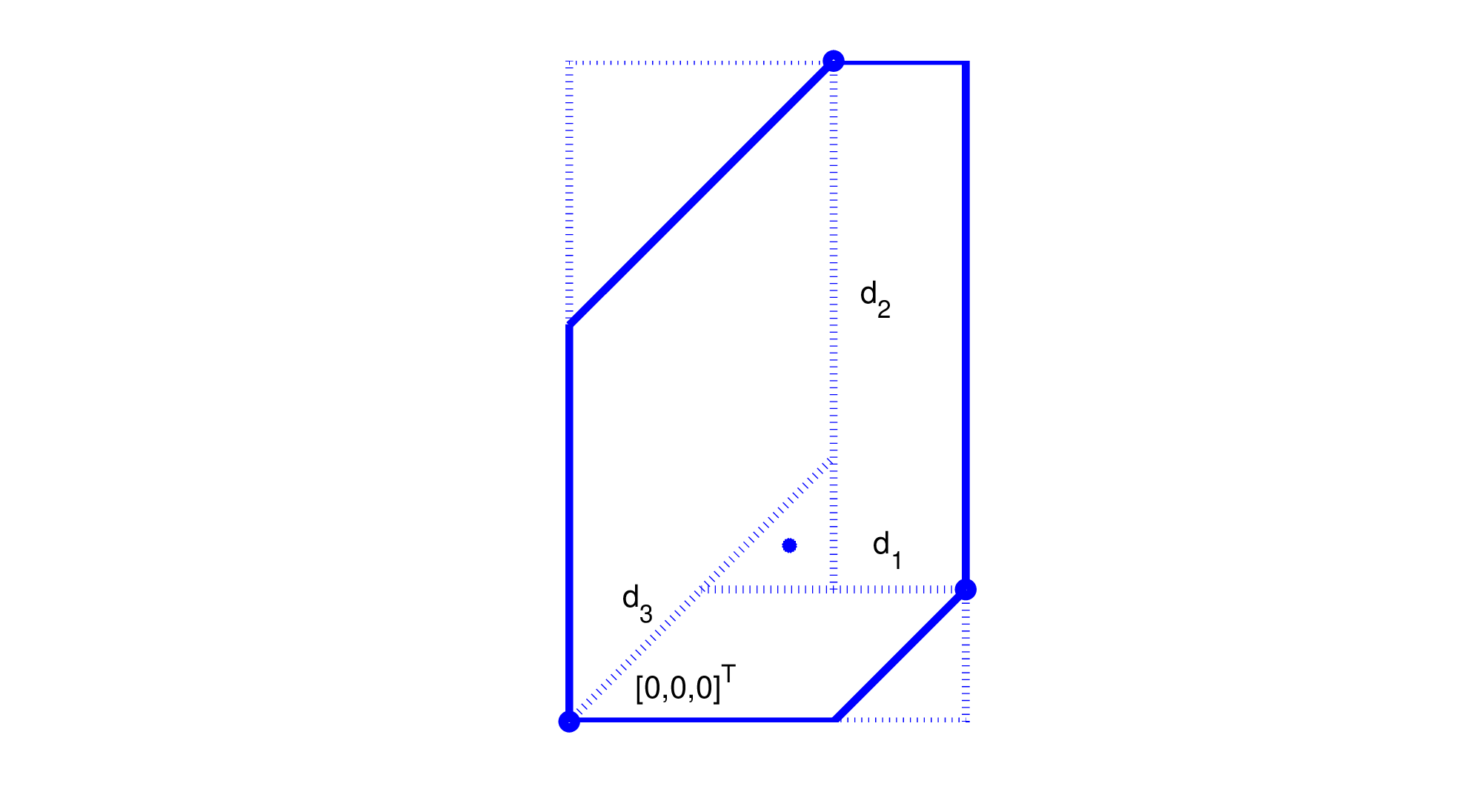}\\
  \caption{Looking for $d,\su d$  in the proof of lemma \ref{lem:canonicalA=A2} , with $t_{21}\le t_{12}$.}
  \label{fig:rectangle_19}
  \end{figure}

  \begin{figure}[H]
 \centering
  \includegraphics[keepaspectratio,width=10cm]{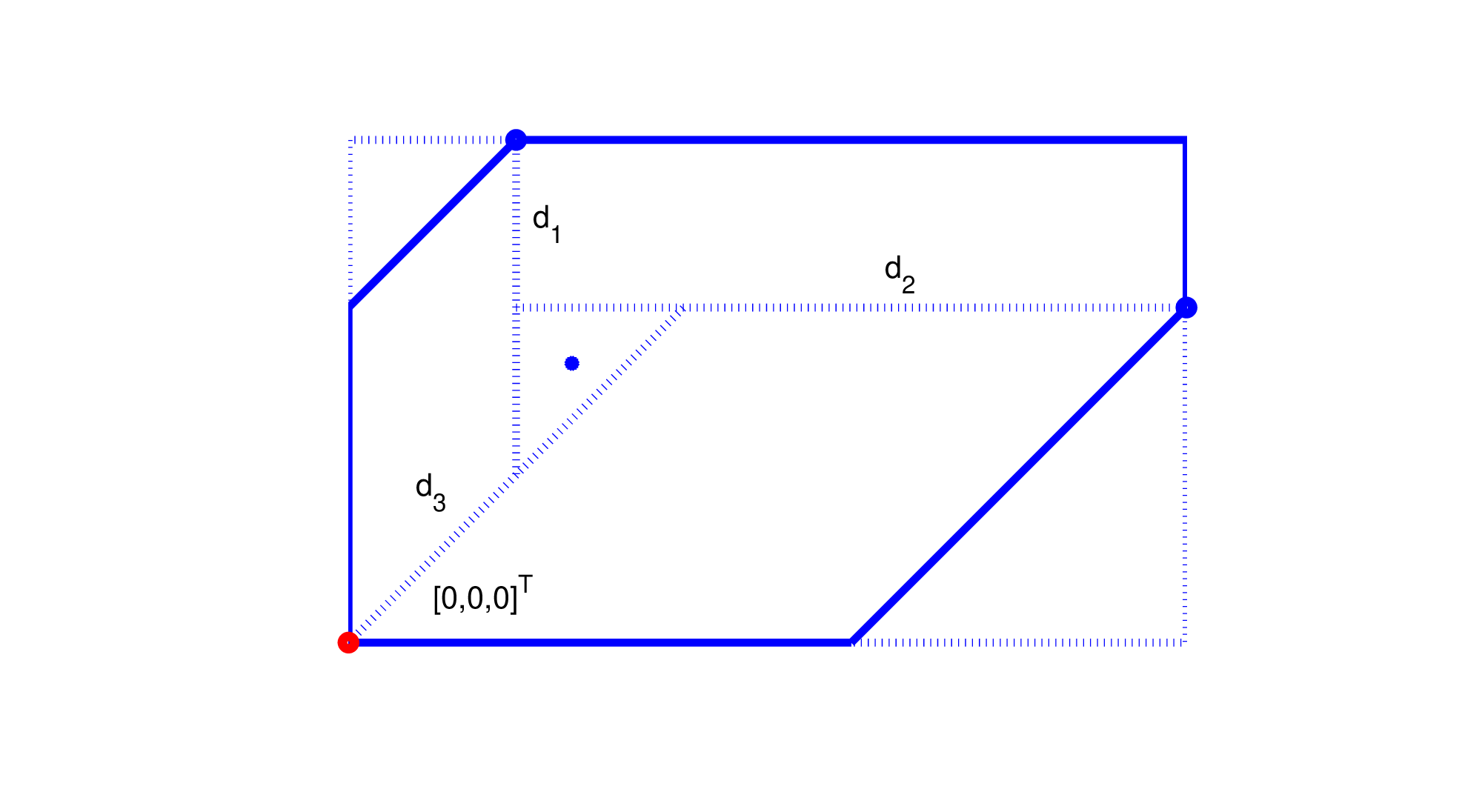}\\
  \caption{Looking for $d,\su d$ in the proof of lemma \ref{lem:canonicalA=A2}, with $t_{21}\ge t_{12}$.}
  \label{fig:rectangle_18}
  \end{figure}

\begin{ex} Suppose that $A$ is idempotent normal with $t_{12}=t_{21}$ and $t_{11}=t_{22}$.
$$A_0=\left[\begin{array}{ccc} t_{11}&t_{12}&0\\t_{12}&t_{11}&0\\0&0&0
\end{array}\right].$$ Then $0\le t_{12}\le t_{11}$ and the new three points shown in figures \ref{fig:rectangle_18} and \ref{fig:rectangle_19} collapse to $[t_{12},t_{12},0]^T$. In this case $d=0$, $d_1=d_2=t_{11}-t_{12}$, $d_3=t_{12}$, $Q=\diag(d_3,d_3,0)$ and $L=Q^{\odot -1}\odot A\odot Q$ satisfies
$$L=\left[\begin{array}{ccc}
0&t_{12}-t_{11}&-t_{12}\\
t_{12}-t_{11}&0&-t_{12}\\
t_{12}-t_{11}&t_{12}-t_{11}&0\\
\end{array}\right], \qquad L_0=\left[\begin{array}{ccc}
t_{11}-t_{12}&0&-t_{12}\\
0&t_{11}-t_{12}&-t_{12}\\
0&0&0\\
\end{array}\right].$$ Notice that $P_{12}\odot L \odot P_{12}=L$.
\end{ex}


\begin{cor}\label{cor:convex} A good tropical triangle is
 classically convex in $Z=0$.
\end{cor}

\begin{proof} Let $\TT_A$ be a good tropical triangle, for some matrix $A$.
By the paragraph after (\ref{eqn:normalizacion}), a translation allows us to assume that $A$ is normal. By lemmas \ref{lem:A=A2}  and \ref{lem:canonicalA=A2},
we can assume that  $A=L(d,\su d)$, for some $d,\su d\ge0$.
 By lemma \ref{lem:new}, $\TT_A$ is determined by the columns of the matrices $A$ and $-A^T$ and,  working in $Z=0$, we must look at the matrices $A_0$ and $(-A^T)_0$ shown above
in (\ref{eqn:LLLcero}) and (\ref{eqn:LLLTmenoscero}). Convexity immediately follows; see figure \ref{fig:normal_idempot_11}.
\end{proof}

\begin{lem}\label{lem:antennas}
Given  $d,\su d\in\R$,  $d\ge0$, $-d\le d_j$,
for $j=1,2,3$, set $A=L(d,\su d)$.
The following are equivalent:
\begin{enumerate}
\item $A\neq A^{\odot2}$,
\item $d_j<0$, for some $j=1,2,3$,
\item in $Z=0$, $\TT_A$ is not classically convex.
\end{enumerate}
\end{lem}

\begin{proof} To check for convexity in $Z=0$, consider  the matrix $A_0$ in expression (\ref{eqn:LLLcero}). Say, $j=1$. If  $d_1<0$, then  any point in the classical segment $\overline{p,q}$, with $p=[d+d_1, 2d+d_1+d_2,0]^T$ and $q=\col(A_0,2)$,  prevents  $\TT_A$ from being  convex; see figure \ref{fig:d_1_non_pos_13}.
\end{proof}

Denote by $a$ the \emph{classical directed segment from $p$ to $q$, including $q$ and excluding $p$}, as in the proof above.
We will say that $a$ is an \emph{antenna} of $\TT_A$. The integer length of $a$ is $-d_1$ and the direction of $a$   is north--east. For $d_2<0$ (resp. $d_3<0$) we would get an antenna pointing west (resp. south).

\begin{figure}[H]
 \centering
  \includegraphics[keepaspectratio,width=12cm]{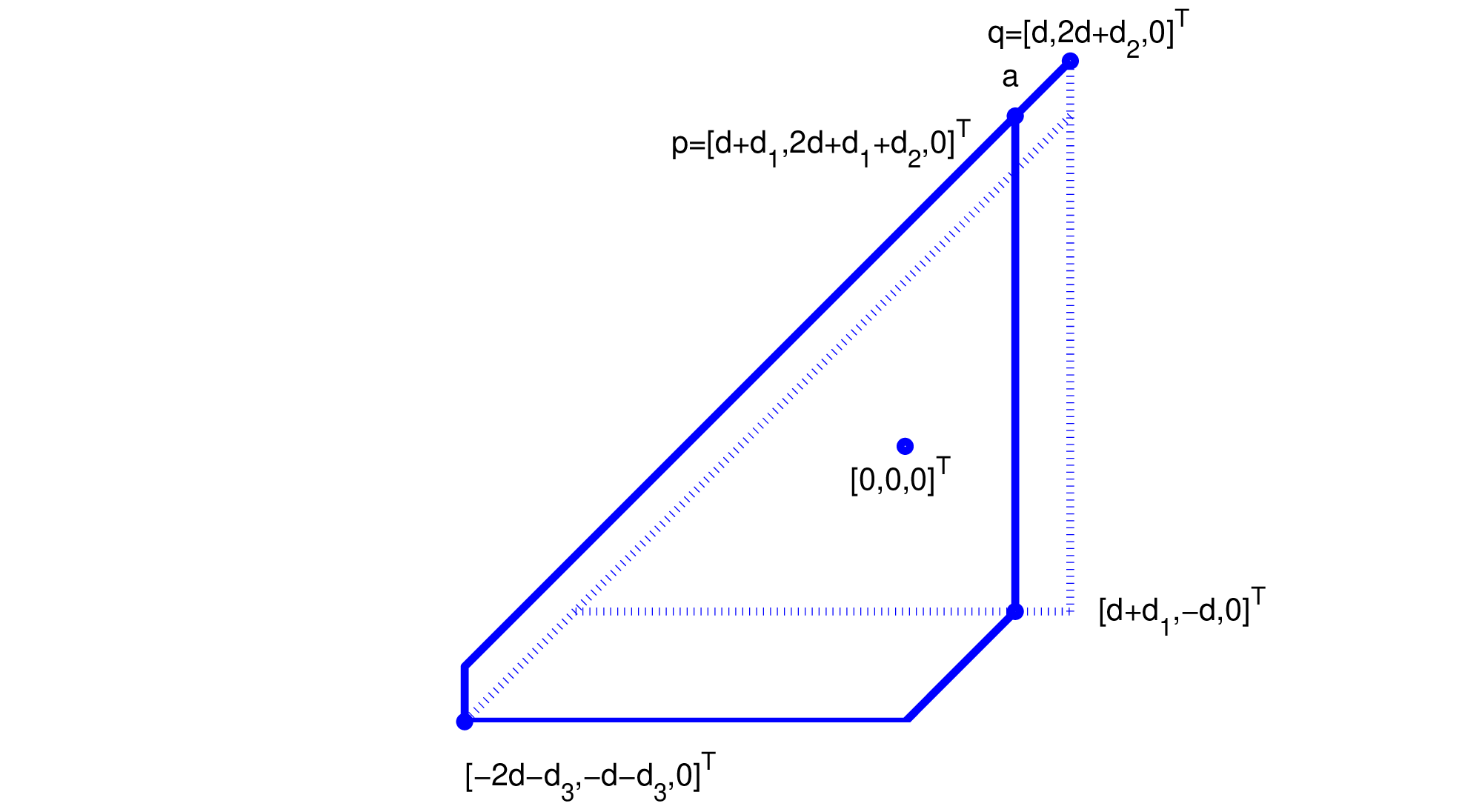}\\
  \caption{Tropical triangle with an antenna  $a$ due to  $d_1<0$.}
  \label{fig:d_1_non_pos_13}
  \end{figure}

\m In the hypothesis of the former lemma,  $\TT_A$  admits a \emph{cell decomposition}\label{dfn:cell_decom_triangle} having, at most, 13 cells (relatively open sets), and this is the generic case:
\begin{itemize}
\item one two--dimensional cell,
\item six one--dimensional cells,
\item six zero--dimensional cells.
\end{itemize}
Any one--dimensional cell disrupting the convexity of $\TT_A $ in $Z=0$ gives rise to an \emph{antenna},\label{dfn:antenna} as in \cite{Cohen}. Each antenna is the union of a one--dimensional  cell and a zero--dimensional  cell.  The union of points in the antennas of $\TT_A $ will be denoted $\ant(\TT_A )$. In  lemma \ref{lem:antennas}, we have shown that
each  $d_j<0$ yields   an antenna in $\TT_A$. The integer length  of this antenna is  $|d_j|$ and $|d_j|\le d$.

Notice that this cell decomposition is not the same as the one associated to $A$, as defined on p. \pageref{dfn:cell_decom} and \pageref{dfn:cell_decom2}.

\begin{cor}\label{cor:ant}
Given $d,\su d\in\R$, $d\ge0$, $-d\le d_j$, for $j=1,2,3$,
set $A=L(d,\su d)$. Then $\TT_{A^{\odot2}}=\TT_A\setminus \ant(\TT_A)$, in $Z=0$.
\end{cor}

\begin{proof} We know that $A$ is normal. If   $A=A^{\odot2}$, we just have corollary \ref{cor:convex}. Otherwise $A^{\odot2}=(A^{\odot2})^{\odot2}$, by lemma \ref{lem:normal}, so that  $\TT_{A^{\odot2}}$  is convex in $Z=0$, by lemma \ref{lem:A=A2} and corollary \ref{cor:convex}.
Now we compute
\begin{equation}\label{eqn:A2}
A^{\odot2}=\left[\begin{array}{ccc}
0&-d-d_2&-2d-d_3-d_2^-\\
-2d-d_1-d_3^-&0&-d-d_3\\
-d-d_1&-2d-d_2-d_1^-&0\\
\end{array}\right],
\end{equation}
\begin{equation}
\label{eqn:A2cero}
(A^{\odot2})_0=\left[\begin{array}{ccc}
d+d_1&d+d_1^-&-2d-d_3-d_2^-\\
-d-d_3^-&2d+d_1+d_2&-d-d_3\\
0&0&0\\ \end{array}\right].
\end{equation}
where $s^-=\min\{s,0\}$.
If, say, $d_1<0$, then 
$\col((A^{\odot2})_0,2)=[d+d_1,2d+d_1+d_2,0]^T$ and $\col(A_0,1)=[d+d_1,-d,0]^T$ so that both points lie on the classical  line $X=d+d_1$, meaning that the antenna in $\TT_A $ caused by the inequality $d_1<0$ no longer appears in $\TT _{A^{\odot2}}$.
\end{proof}

The former corollary tells us that \emph{squaring the normal matrix $A=L(d,\su d)$ corresponds to chopping off the antennas of $\TT_A$, if any}. The tropical triangle $\TT _{A^{\odot2}}$ will be called the \emph{soma} \label{dfn:soma} of $\TT_A $, denoted $\soma(\TT_A )$. Then \begin{equation}\label{eqn:union} \TT_A=\soma(\TT_A )\cup\ant(\TT_A ) \end{equation} is a disjoint union.


Notice that
$A^{\odot2}=0$ does not imply $A=0$, even if $A$ is normal. For example,
\begin{equation}\label{eqn:alpha}
A=\left[\begin{array}{ccc}
0&0&-\gamma\\-\alpha&0&0\\0&-\beta&0\\
\end{array}\right],
\end{equation} with $\alpha,\beta,\gamma\ge0$.

\m 

We know that  the antennas (if any) of $\TT_A$ have integer length $|d_j|\le d$,  when $A=L(d,\su d)$ with $d\ge0$ and  $-d\le d_j$. \label{asser:antennas_length}
But, there exist tropical triangles  with antennas of arbitrary length. Moreover, notice the way that antennas wrap around the two dimensional part of a triangle $\TT_A$,  for $A=L(d,\su d)$ with $d>0$, $0<|d_j|<d$, $j=1,2$ and compare with  the essentially different way that antennas wrap around the two dimensional part of the triangle $\TT_B$ (see figure \ref{fig:erratum_1}), for
\begin{equation}\label{eqn:B}
B=\left[\begin{array}{rrr}
0&-5&0\\-7&0&0\\-6&-1&0
\end{array}\right],\qquad  B_0=\left[\begin{array}{rrr}
6&-4&0\\-1&1&0\\0&0&0
\end{array}\right].
\end{equation}
For these two reasons,  in order to find a canonical normalization for the matrices describing these triangles, we must consider   matrices more general than $L(d,\su d)$.

Consider
\begin{equation}
F(d,\su d,\su h,g)=\FH,\label{eqn:FH}
\end{equation}
with $d,\su d,\su h,g\ge0$ such that
\begin{description}
\item [condition 1:] $h_{j+1}>0$ implies $d_j=0$,   \label{rem:1}
\item [condition 2:] $g>0$ implies $d=d_1=h_3=0$,   \label{rem:2}
\end{description}
where all subscripts work modulo 3. Notice how the positivity of some of the parameters inhibit the positivity of other parameters.
 For pictures in $Z=0$, we will use
\begin{equation}
F(d,\su d,\su h,g)_0=\FHcero,\label{eqn:FHcero}
\end{equation}
\begin{equation}
(-F(d,\su d,\su h,g)^T)_0=\FHTmenoscero.\label{eqn:FHTmenoscero}
\end{equation}

Notice that
\begin{itemize}
\item $F(d,\su d,0,0,0,0)=L(d,\su d)$.

\item $F=F(d,\su d,\su h,g)$ is  normal and $F^{\odot2}=L(d,\su d)$.

\end{itemize}

\begin{thm}[Lower canonical normalization]\label{thm:canonical_nf}
If $A$ is any $3\times 3$  real matrix, then there exist unique $d,\su d,\su h,g\ge0$ 
and there exist permutation matrices $P,Q$ such that
$F(d,\su d,\su h,g)=P\odot A\odot Q$, where all subscripts work modulo 3.
\end{thm}
\begin{proof}
To prove existence, we may assume that $A$ is normal. If $A=A^{\odot2}$, we take $h_j=g=0$,  for all $j$, and apply  lemma \ref{lem:canonicalA=A2}.
Now  assume that $A\neq A^{\odot2}$.   By the same lemma, we can assume that $A^{\odot2}=L(d,\su d)$, for some $d,\su d\ge0$.
In other words, $A$ is a \emph{tropical square root of $L(d,\su d)$}.
The geometric  meaning of this assertion is  that \emph{$\TT_A$ is obtained from $\TT_{A^{\odot2}}$ by pasting of antennas}. Each antenna of
$\TT_A$ emanate from a vertex of $\TT_{A^{\odot2}}$, i.e., from the point represented by a column of $A^{\odot2}$. At each vertex of $\TT_{A^{\odot2}}$,
an antenna can be glued, at most.

Write
\begin{equation}
p=\col(A^{\odot2},1), \quad q=\col(A^{\odot2},2), \quad r=\col(A^{\odot2},3),
\end{equation}
\begin{equation}
p'=\col(A,1), \quad q'=\col(A,2), \quad r'=\col(A,3).
\end{equation}

\begin{description}
\item [Claim.] For $\TT_{A^{\odot2}}$ to admit  gluing of antennas, it is necessary that the points $p$, $q$, $r$, $-(p\otimes q)$,
$-(q\otimes r)$ and $-(r\otimes p)$ are NOT all different. In particular, $d_j=0$, for some $j$.
\end{description}
Indeed, suppose a small antenna $a$ emanates from $r$; this means that $r'$ is perturbation of $r$. Let us work in $Z=0$. Recalling the rays of an affine
tropical line,
this antenna can point  either west or south. Suppose that $a$ points west. Then
$r'=[-2d-d_3-\epsilon,-d-d_3,0]^T$ and  $\epsilon>0$ is the integer length of $a$.  Consider the slope--one classical line through
$q=[q_1,q_2,0]^T$, with equation
$Y-X=q_2-q_1$, and the slope--zero line through $r=[r_1,r_2,0]^T$, with equation
$Y=r_2$. These lines meet at the point $i=[r_2-q_2+q_1,r_2,0]^T$. Now $r_2-q_2+q_1\le r_1$, since $\TT_{A^{\odot2}}$ is a good triangle (the former
inequality
is just another way of writing
$a_{12}-a_{22}\le a_{13}-a_{23}$ in (\ref{eqn:6})). For  the antenna $a$ to exist, $r'$ must be west of $i$ and for $a$ to emanate from $r$,
it must
happen that $i=r$, or equivalently, $r=-(q\otimes r)$. A simple computation shows that $-(q\otimes r)=[-2d-d_3,-d-d_3+d_2,0]^T$, so that
$d_2=0$; see figure \ref{fig:erratum_2}.

\begin{figure}[H]
 \centering
  \includegraphics[keepaspectratio,width=12cm]{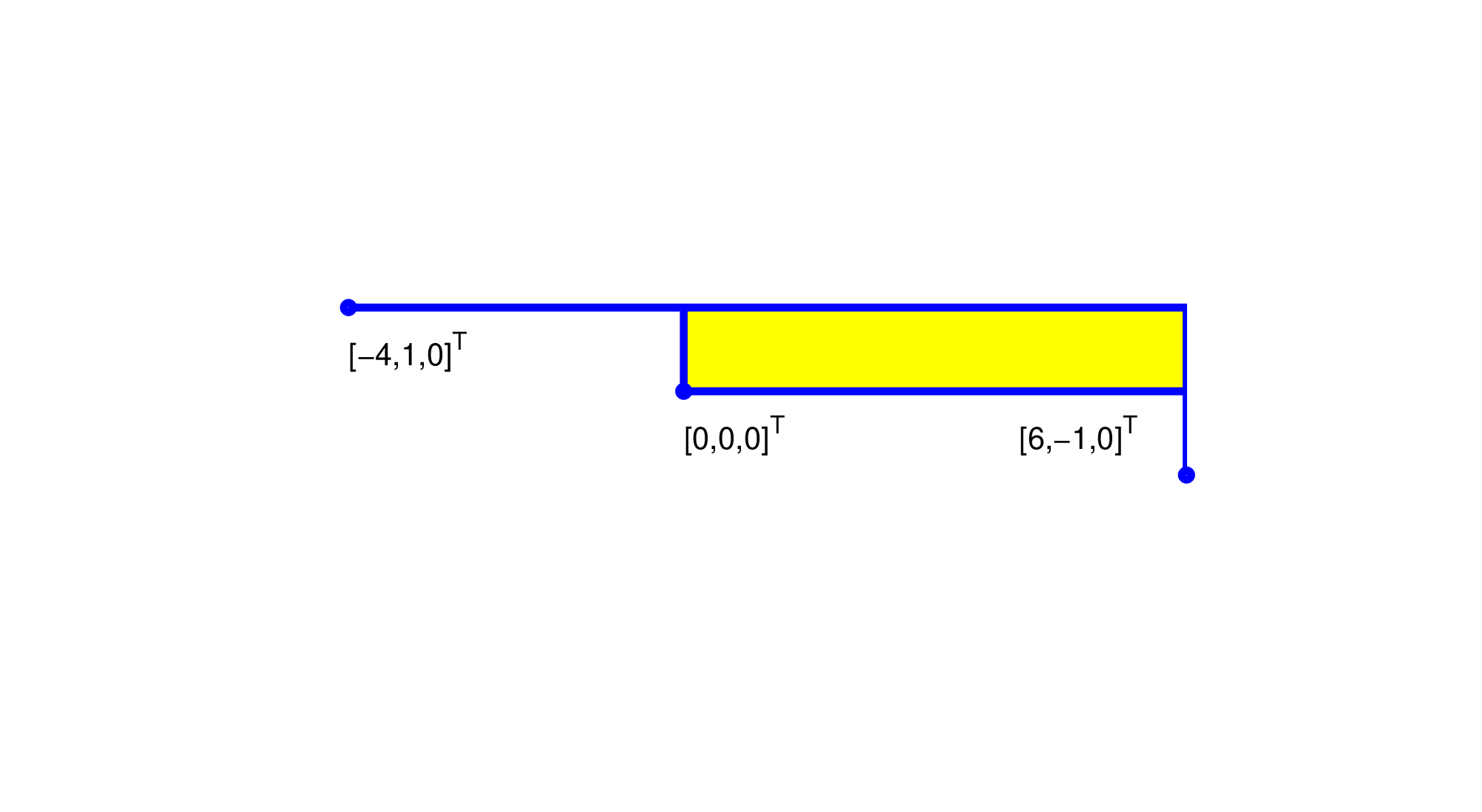}\\
  \caption{The tropical triangle corresponding to matrix $B$ in (\ref{eqn:B}).}
  \label{fig:erratum_1}
  \includegraphics[keepaspectratio,width=12cm]{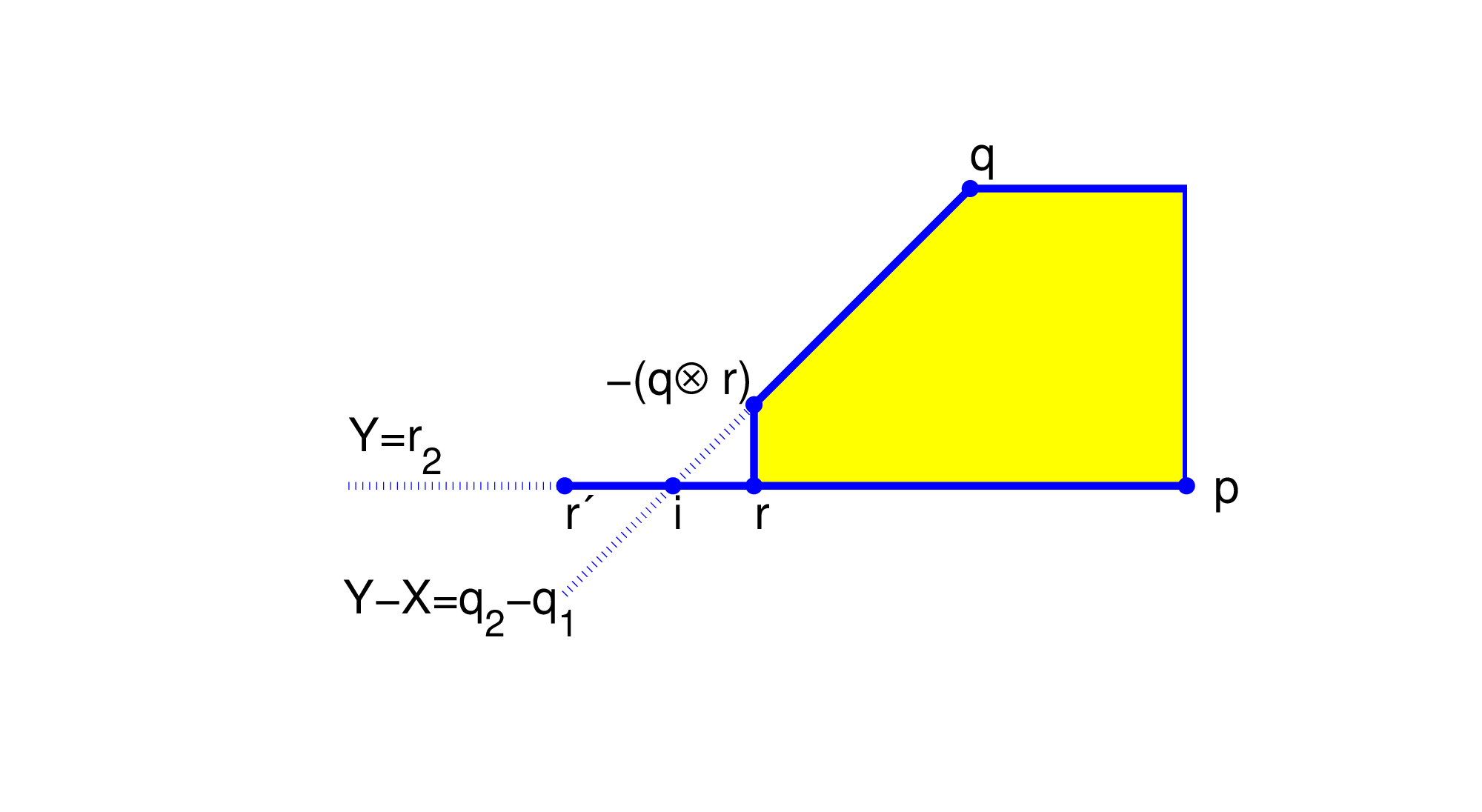}\\
  \caption{The three points $i,r$ and $-(q\otimes r)$ will collapse together.}
  \label{fig:erratum_2}
  \end{figure}

Suppose now that $a$ points south.
 Then $r'=[-2d-d_3,-d-d_3-\epsilon,0]^T$ and  $\epsilon>0$ is the integer length of $a$.  In this case the points $r$ and $-(r\otimes p)$
coincide.
A simple computation shows that
$-(r\otimes p)=[d+d_1-d_3,-d-d_3,0]^T$. It follows that $d_1=-3d$, whence $d=d_1=0$. In particular, $p=[0,0,0]^T$.

The proof of the claim is similar, when $a$  emanates from $p$ or $q$.

\begin{itemize}
\item Assume $A^{\odot2}=0$, i.e.,  $d=d_1=d_2=d_3=0$. Then, all the columns  of $A^{\odot2}$ represent the same point (the origin) and $\TT_A$ is the
union of the
origin and, at most, three antennas, of integer lengths $\su h\ge0$.
Take then $g=0$ to get $A=F(0,0,0,0,\su h,0)$. 

\item Assume now that the columns of $A^{\odot2}$ represent three different points.
Then, either $d\neq0$ or $d=0$ and $d_j\neq0\neq d_{j+1}$, for some $j$.

We consider here matrices a bit more general than $F(d,\su d,\su h,g)$.
Write $G(d,\su d,\su h,\su g)=$
\begin{equation}
\GH
\end{equation}
such that  $d,\su d,\su h,\su g\ge0$ and
\begin{enumerate}
\item $h_{j+1}>0$ implies $d_j=g_{j+1}=0$,   \label{rem:11}
\item $g_j>0$ implies $d=d_{j+1}=h_j=0$,   \label{rem:21}
\end{enumerate}
where all subscripts work modulo 3. A simple computation shows that $L(d,\su d)$  equals $G(d,\su d,\su h,\su g)^{\odot2}$.

Suppose that $A=G(d,\su d,\su h,\su g)$.
\begin{enumerate}

\item Say, $\TT_A$ has an antenna  $a$  of integer length $\epsilon$ emanating from the point represented by a column of $A^{\odot2}$. Say this column is $r$.
\begin{enumerate}
\item If $a$ points south,
 then   $d=d_1=0$ and $p=[0,0,0]^T$. Take $g_3=\epsilon$ and $h_3=0$.

\item If $a$ points west, then $d_2=0$. Take $g_3=0$ and $h_3=\epsilon$.
\end{enumerate}

\item Say, $\TT_A$ has two antennas  $a$ and $b$, and $a$ is emanating from  $r$. Looking at the previous item, we see that only
three cases are possible:
\begin{enumerate}
\item  If $g_3>0$, $d=d_1=h_3=0$ and $h_2>0$, $d_1=g_2=0$,  so that $h_2$ is the integer length  of $b$. The antenna $a$  points south and $b$ points north--east from $q$; see figure \ref{fig:erratum_3}, left,
\item If $h_3>0$, $d_2=g_3=0$ and $h_2>0$, $d_1=g_2=0$,  so that $h_2$ is the integer length  of $b$.   The antenna $a$  points west and $b$  points north--east from $q$; see figure \ref{fig:erratum_3}, center,
\item If $h_3>0$, $d_2=g_3=0$ and $h_1>0$, $d_3=g_1=0$,  so that $h_1$ is the integer length  of $b$. The antenna $a$  points west and $b$  points south from $p$; see figure \ref{fig:erratum_3}, right.
\end{enumerate}
\begin{figure}[H]
 \centering
  \includegraphics[keepaspectratio,width=13cm]{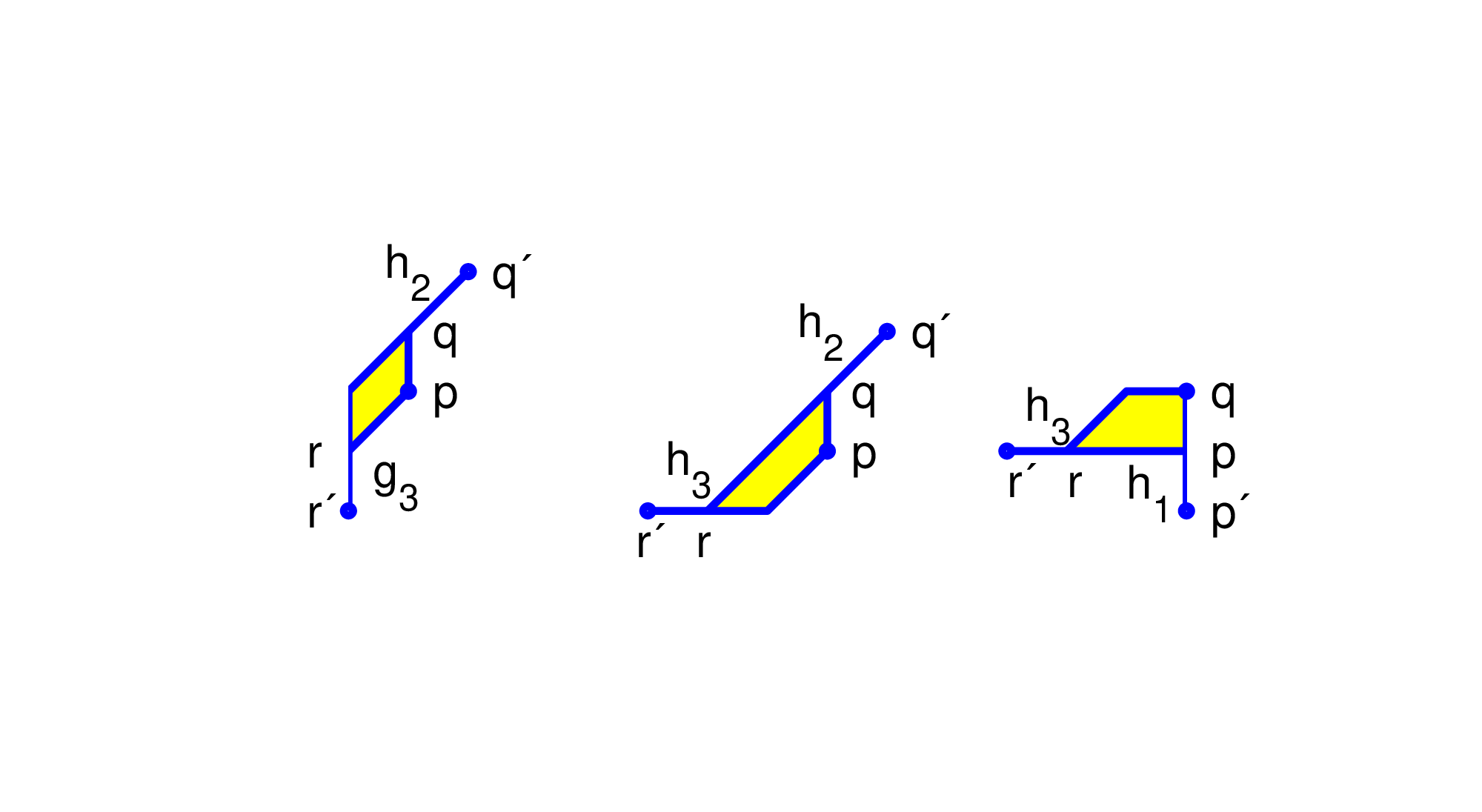}\\
  \caption{All tropical triangles with two antennas (up to change of coordinates).}
  \label{fig:erratum_3}
  \includegraphics[keepaspectratio,width=13cm]{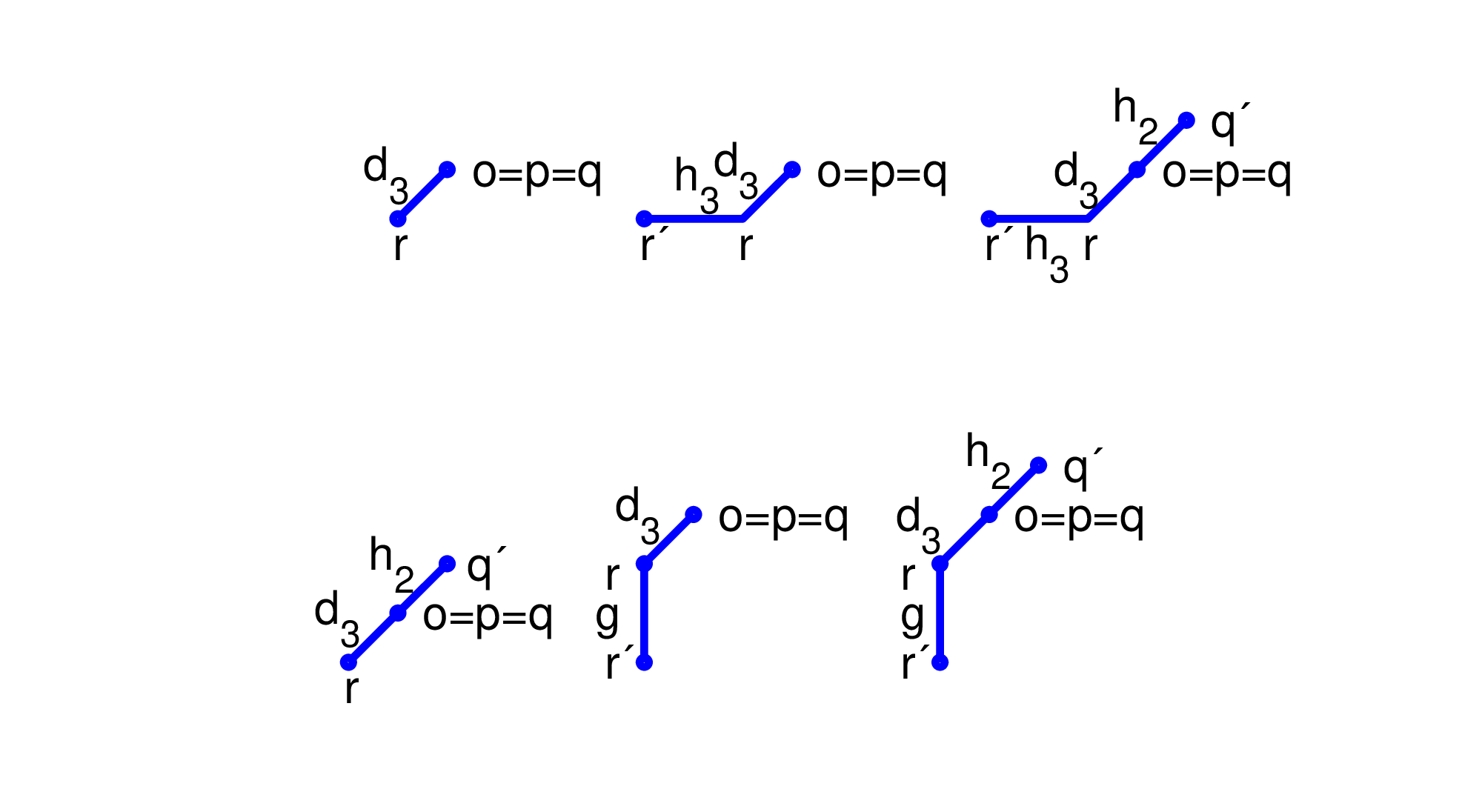}\\
  \caption{All tropical triangles having 1--dimensional somas (up to change of coordinates).}
  \label{fig:erratum_4}
  \end{figure}

If we had $g_3>0$ and $g_2>0$, then $d=d_1=d_3=0$, contradicting  that the columns of $A^{\odot2}$
represent three different points. 
Similarly, if $g_3>0$ and $g_1>0$, or if  $g_2>0$ and $g_1>0$.

\item Say, $\TT_A$ has three antennas. Then, by the previous item, the only possibility is $\su h>0$, so that $d_j=g_{j+1}=0$, for all $j$.
\end{enumerate}

We have just proved that $g_j>0$ is possible for only one value of  $j=1,2,3$. Then, a change of coordinates allows us to assume $j=3$, and write $g_3=g$,  so that
 $G(d,\su d,\su h,0,0,g)=F(d,\su d,\su h, g)$.

\item Assume now that the columns of $A^{\odot2}$ represent two different points. This case can be viewed as a degeneration of
the previous cases, as $d>0$ tends to zero.  Say $d=d_1=d_2=0$ and $d_3\neq0$. Then   $p=q$ is the origin and $h_1=0$. Here
$\soma(\TT_A)$  reduces to a classical segment \label{rem:segment}; see figure \ref{fig:erratum_4}.
\end{itemize}

So far, we have proved that given any $3\times 3$ matrix $A$, there exist permutation matrices $P,Q$ such that
$$F(d,\su d,\su h, g)^{\odot2}=L(d,\su d)=P\odot A^{\odot2}\odot Q.$$ It turns out that the same matrices $P,Q$ provide
$$F(d,\su d,\su h, g)=P\odot A\odot Q.$$
The uniqueness of $d,\su d,\su h,g$ follows from the geometric meaning of these parameters.
\end{proof}

\begin{cor}\label{cor:sqrt}
 $F(d,\su d,\su h,g)$ is a
tropical square root of $L(d,\su d)$.\qed
\end{cor}

\begin{ex}\label{ex:cohen1}
Let us compute the lower canonical normalization of   $$A=A_0=\left[
\begin{array}{ccc}
0&1&3\\
0&3&4\\
0&0&0\\
\end{array}\right],$$ the example in \cite{Cohen}, p. 409. Consider the  matrices $U =\diag (-1,0,0)$,
$$Q=\left[\begin{array}{ccc}
-\infty&-\infty&0\\
0&-\infty&-\infty\\
-\infty&0&-\infty
\end{array}\right]$$ and obtain
$$N_0=U\odot P_{12}\odot A\odot Q=\left[\begin{array}{ccc}
2&3&-1\\
1&3&0\\
0&0&0
\end{array}\right].$$ Then
$$N=N_0\odot \diag (-2,-3,0)=\left[\begin{array}{ccc}
0&0&-1\\
-1&0&0\\
-2&-3&0
\end{array}\right]$$ is normal.
 Then
$$N^{\odot2}=\left[\begin{array}{ccc}0&0&0\\-1&0&0\\-2&-2&0
\end{array}\right],\qquad (N^{\odot2})_0=\left[\begin{array}{ccc}
2&2&0\\
1&2&0\\
0&0&0
\end{array}\right].$$
In figure \ref{fig:ejem_cohen_15}, we find, from left to right, the tropical triangles corresponding to the matrices $A, P_{12}\odot A, U\odot P_{12}\odot A$, $U\odot P_{12}\odot A\odot Q$  and $N$ (the last three matrices yield the same triangle),  all in $Z=0$.
Here $3d=1$, $d_1=d_2=0$ and $d_3=1$ so that
the lower canonical normalization of $N^{\odot2}$
is
$$L\left(1/3,0,0,1\right)=
\left[\begin{array}{ccc}
0&-1/3&-5/3\\
-2/3&0&-4/3\\
-1/3&-2/3&0
\end{array}\right]=S^{\odot -1}\odot N^{\odot2}\odot S,$$  with $S=\diag(5/3,4/3,0)$.
Thus, the lower canonical normalization of $A$ is
$$S^{\odot -1}\odot N\odot S=F\left(1/3,0,0,1,0,1,1,0\right)=\left[\begin{array}{ccc}
0&-1/3&-8/3\\
-2/3&0&-4/3\\
-1/3&-5/3&0
\end{array}\right].$$
We have $h_1=0$, $h_2=h_3=1$.
For pictures in $Z=0$, we consider
$$L\left(1/3,0,0,1\right)_0=\left[\begin{array}{ccc}
1/3&1/3&-5/3\\
-1/3&2/3&-4/3\\
0&0&0
\end{array}\right]$$ and
$$F\left(1/3,0,0,1,0,1,1,0\right)_0=\left[\begin{array}{ccc}
1/3&4/3&-8/3\\
-1/3&5/3&-4/3\\
0&0&0
\end{array}\right].$$ In figure \ref{fig:ejem_cohen_16} we see the triangles corresponding to the matrix $N^{\odot2}$ and its lower canonical normalization, while in figure \ref{fig:ejem_cohen_17} we see the triangles of the matrix $N$ and its lower canonical normalization. Notice that  the matrices $S^{\odot-1}$ and $S$ provide the lower canonical normalization of $N^{\odot2}$ and also of $N$. 
\end{ex}

\begin{figure}[H]
 \centering
  \includegraphics[keepaspectratio,width=12cm]{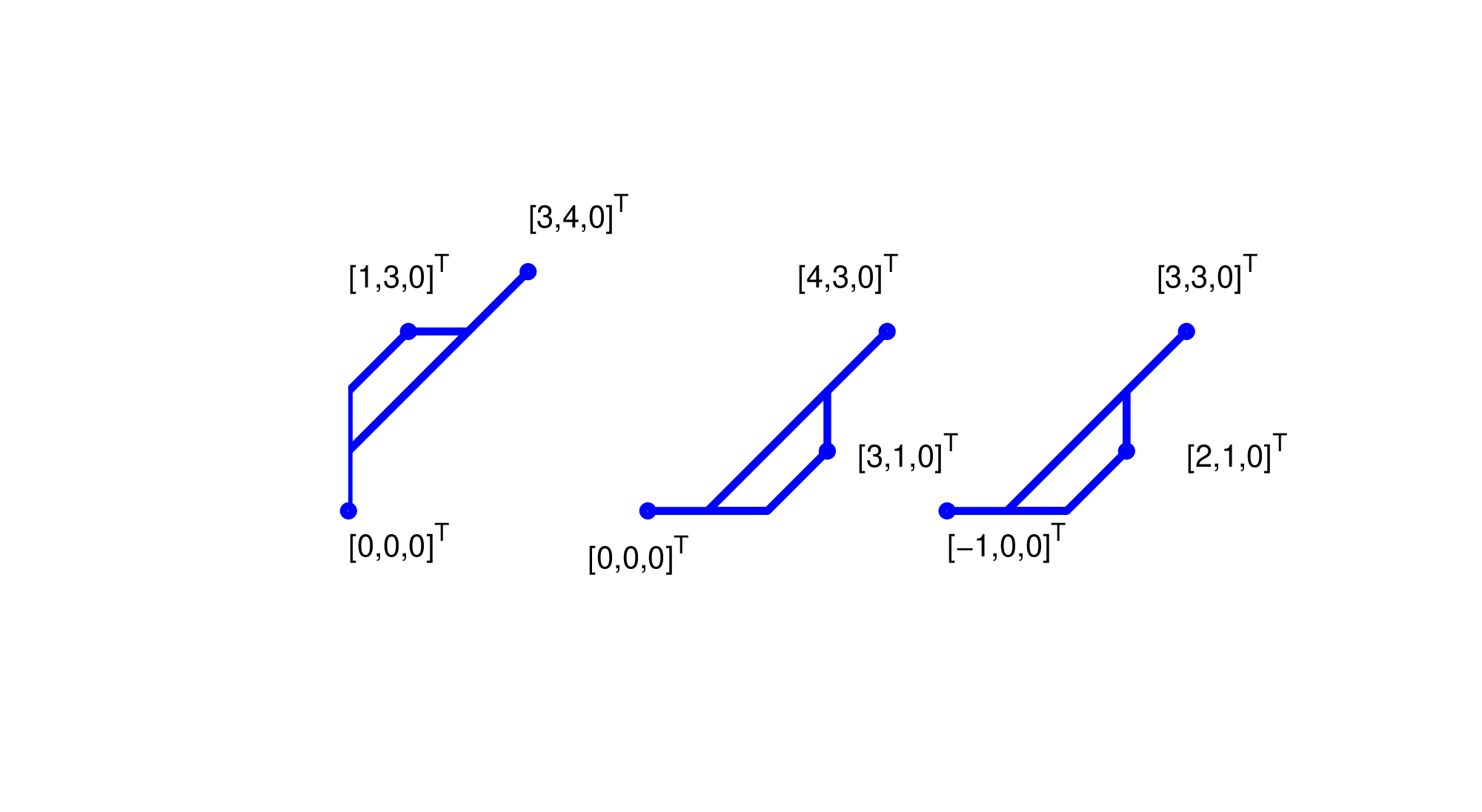}\\
  \caption{Tropical triangles $\TT_A$, $\TT_{P_{12}\odot A}$ and $\TT_N$.}
  \label{fig:ejem_cohen_15}
  \end{figure}

\begin{figure}[H]
 \centering
  \includegraphics[keepaspectratio,width=10cm]{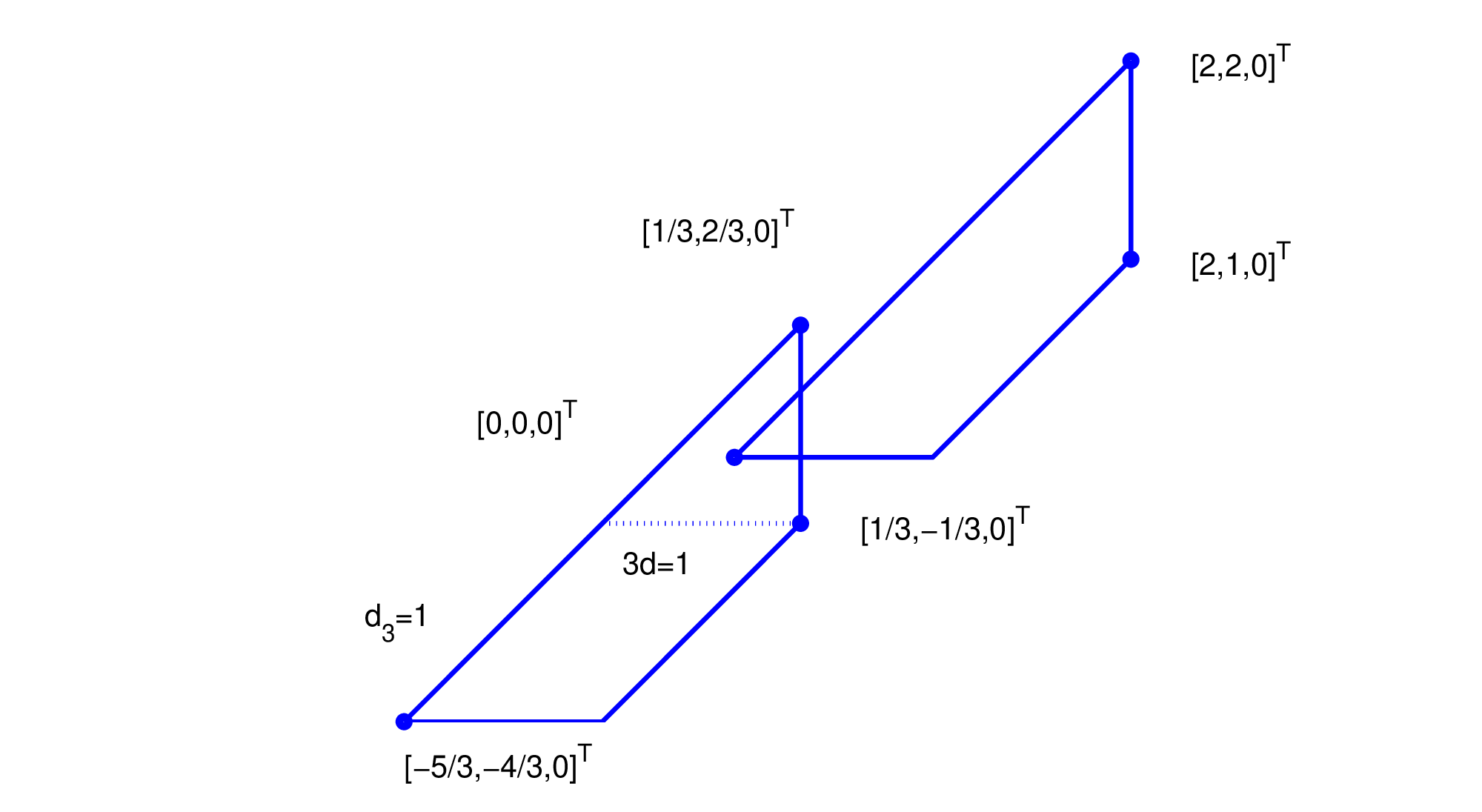}\\
  \caption{Tropical triangle corresponding to $N^{\odot2}$ and to its lower canonical normalization.}
  \label{fig:ejem_cohen_16}
  \end{figure}

\begin{figure}[H]
 \centering
  \includegraphics[keepaspectratio,width=10cm]{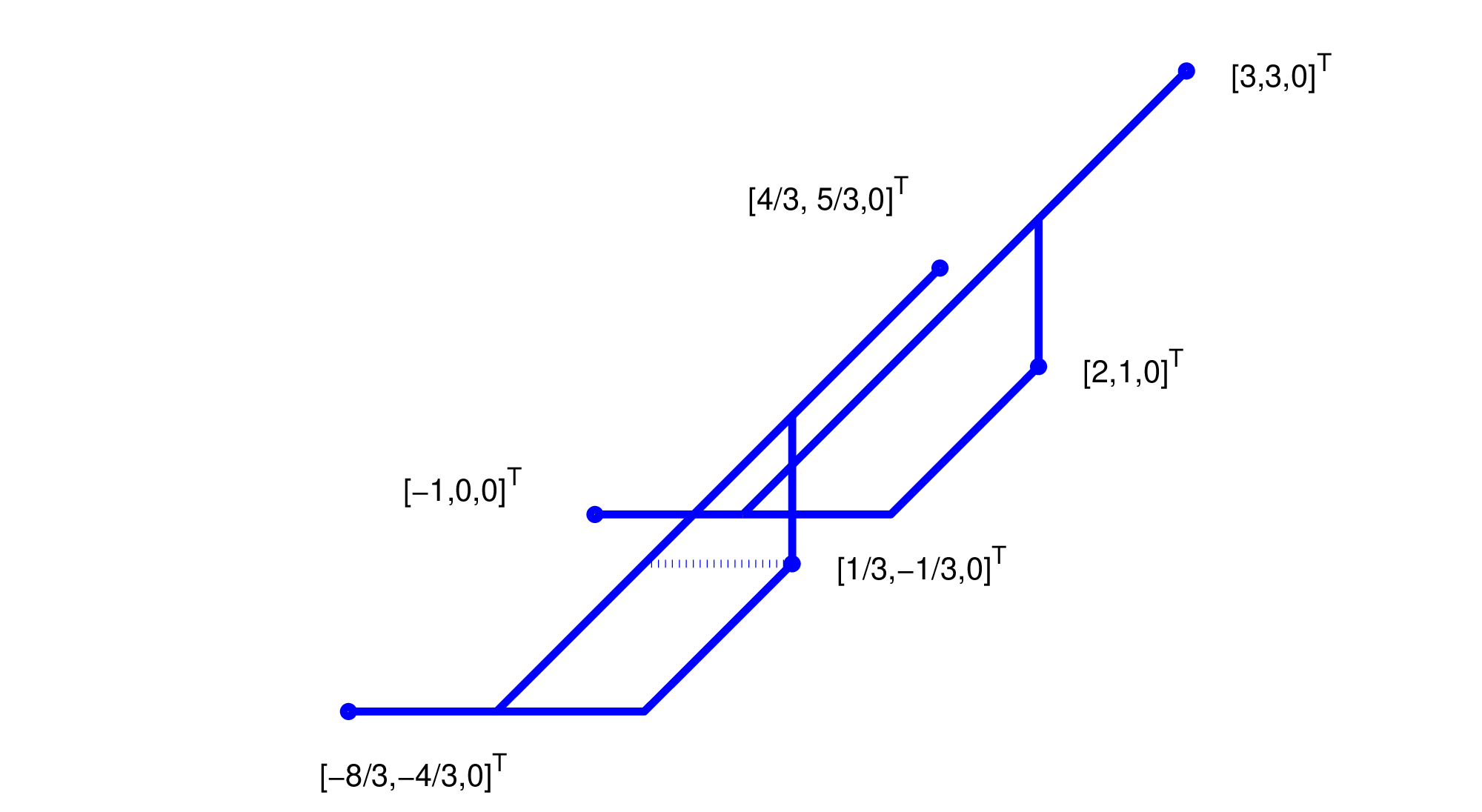}\\
  \caption{Tropical triangle corresponding to $N$ and to its lower canonical normalization.}
  \label{fig:ejem_cohen_17}
  \end{figure}

\m Consider  $A=F(d,\su d,\su h,g)$. Now, a definition of \emph{soma and antennas of $\TT_A
$} can be given, as in  p. \pageref{dfn:antenna}. The soma
of $\TT_A$ is $\TT_{L(d,\su d)}$.   The antennas of $\TT_A
$  have  tropical length $h_j$, if $h_j>0$ or $g$, if $g>0$.

\m For any  $3\times 3$  matrix  $A$ let   $N=F(d,\su d, \su h,g)$ be the lower canonical normalization of $A$. There exist permutation matrices $P,Q$ such that $N=P\odot A\odot Q.$ The map $f_P$ is a translation, so that
the  triangle $\TT_A=f_{P^{\odot-1}}(\TT_N)$ is just a translated of $\TT_N$.
Then we define the \emph{soma and antennas of $\TT_A$} as  follows:
\begin{equation}
\soma(\TT_A)=f_{P^{\odot-1}}(\soma \TT_N),\qquad \ant(\TT_A)=f_{P^{\odot-1}}(\ant \TT_N),\label{eqn:soma_antennas}
\end{equation}
so that the decomposition (\ref{eqn:union}) holds true, for any $A$.

%
%

\m The following theorem is a  simple  geometric characterization of normality.

\begin{thm}\label{thm:normalZ=0}
The  $3\times 3$ matrix  $A$ is normal if and only if
 the origin belongs to $\soma(\TT_A)$, in $Z=0$.
\end{thm}
\begin{proof} Recall expression (\ref{eqn:col_0}). If  $A$ is normal then $A^{\odot2}$ is idempotent normal, so that $\soma(\TT_A)$ is  a good triangle. By expression (\ref{eqn:good1}), in $Z=0$ we have  $\soma(\TT_A)=\TT$, for some $a,b,c,d,e,f\in\R$. Then
\begin{equation}
c-f\le x\le  d-e,\ a+e\le y\le b+f,\ c-b\le y-x\le  d-a.
\end{equation}
If we write
$$\oa=a\oplus (c-f),\quad  \oc=c\oplus (a+e),\quad  \oe=e\oplus (c-b),$$
$$\ob=b\oplus'(d-e),\quad\od=d\oplus'(b+f),\quad\of=f\oplus'(d-a),$$
then
\begin{equation}\label{eqn:good2}
\soma(\TT_A)=\{(x,y)\in \R^2: \oa\le x\le  \ob,\ \oc\le y\le \od,\ \oe\le y-x\le  \of\}
\end{equation}
and $\soma(\TT_A)$ is the set of points spanned by the columns of the idempotent matrix $B$, where
\begin{equation}
B=\left[\begin{array}{ccc}
0&-\of&\oa\\\oe&0&\oc\\-\ob&-\od  &0
\end{array}\right],\quad B_0=\left[\begin{array}{ccc}
\ob&\od-\of&\oa\\\ob+\oe&\od&\oc\\0&0&0
\end{array}\right].
\end{equation}
Then $B=A^{\odot2}$, whence
$\oa,\oc,\oe\le0\le \ob,\od,\of$, by normality.  Clearly the origin belongs to $\soma(\TT_A)$ in $Z=0$.

Conversely, if $\soma(\TT_A)$ is a point (the origin), then $A^{\odot2}=0=L(0,0,0,0)$, whence $A=F(0,0,0,0,\su h,g)$, for some $\su h,g\ge0$, showing that $A$ is normal.
On the other hand, if $\soma(\TT_A)$ has dimension $\ge1$ and contains the origin, then  $\col((A^{\odot2})_0,j)\in \overline{C^0_j}$, for $j=1,2,3$. Now, gluing antennas (at most three, and in prescribed directions: north--east, west and south) to $\soma(\TT_A)$ provides $\TT_A$. The end points of these antennas are precisely the columns of $A_0$ and $\col(A_0,j)\in \overline{C^0_j}$, for $j=1,2,3$, showing normality of $A$.
\end{proof}

%

\begin{ex}\label{ex:cohen3}
Let us look at   example \ref{ex:cohen1} again, (see fig. \ref{fig:ejem_cohen2}). A normalization of $A$
is $N=P\odot A\odot Q$, where $P=\diag(-2,-3,0)$,
$$N=\left[
\begin{array}{ccc}
0&-1&-2\\
0&0&-3\\
-1&0&0\\
\end{array}\right], \qquad
Q=\left[\begin{array}{ccc}-\infty&-\infty&0\\-\infty&0&-\infty\\-1&-\infty&-\infty\end{array}\right].$$
We have
$$N_0=\left[\begin{array}{ccc}1&-1&-2\\1&0&-3\\0&0&0\end{array}\right],\qquad N^{\odot2}=\left[
\begin{array}{ccc}
0&-1&-2\\
0&0&-2\\
0&0&0\\
\end{array}\right]=(N^{\odot2})_0.$$
Another normalization of $A$ is $N=P\odot A\odot Q,$ where $P=\diag(-1,-3,0)$,
$$N=\left[
\begin{array}{ccc}
0&0&-1\\
-1&0&-3\\
-2&0&0\\
\end{array}\right],\qquad Q=\left[\begin{array}{ccc}-\infty&-\infty&0\\-\infty&0&-\infty\\-2&-\infty&-\infty\end{array}\right];$$
in this case,
$$N_0=\left[\begin{array}{ccc}2&0&-1\\1&0&-3\\0&0&0\end{array}\right],\ N^{\odot2}=\left[
\begin{array}{ccc}
0&0&-1\\
-1&0&-2\\
-1&0&0\\
\end{array}\right],\ (N^{\odot2})_0=\left[
\begin{array}{ccc}
1&0&-1\\
0&0&-2\\
0&0&0\\
\end{array}\right].$$
\end{ex}

\begin{figure}[H]
 \centering
  \includegraphics[keepaspectratio,width=12cm]{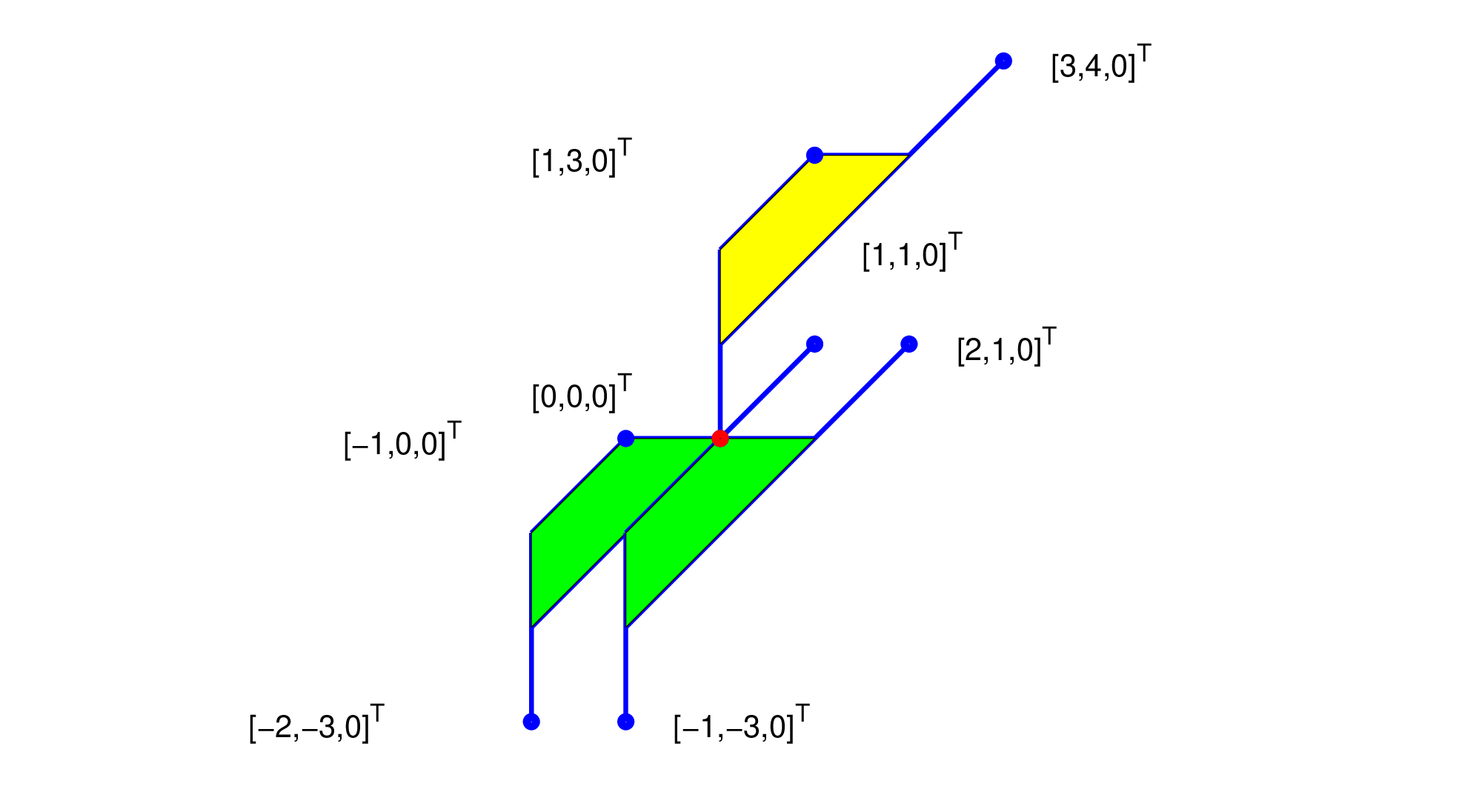}\\
  \caption{Tropical triangles $\TT_A$ and  $\TT_N$, for  two different normalizations $N$.}
  \label{fig:ejem_cohen2}
  \end{figure}

\m Next, we procede to define soma and antennas of a co--spanned tropical triangle.
Regarding co--span, we choose to work with \emph{co--normal} matrices, i.e., matrices $A$ having non--negative entries and zero diagonal.
We can achieve a lower canonical co--normalization and then define \emph{soma and antennas of a co--spanned tropical triangle}, in a similar fashion  to theorem \ref{thm:canonical_nf} and definition in p. \pageref{eqn:soma_antennas}.
Since antennas of spanned triangles grow only in three directions (north--east, west and south), then  antennas of co--spanned triangles grow only in three directions (south--west, east and north).

\m
For a given matrix $A$, consider the triangle $\TT^{A}=\cospan(-A^T)$. Then
 \begin{equation}\label{eqn:union} \TT^A=\soma(\TT^A )\cup\ant(\TT^A ) \end{equation} is a disjoint union.
Suppose that the two--dimensional cell $B^A$, as defined on p. \pageref{dfn:cell_decom}, is non--empty. Then
\begin{equation}
\overline{B^A}=\soma(\TT^A).\label{eqn:B^A}
\end{equation}
On the other hand, $B^A$ is empty if and only if $\soma(\TT^A)$ reduces to a segment or point.

\m Now, what is the relationship  between $\TT_A$ and $\TT^A$,  $\soma (\TT_A)$ and $\soma(\TT^A)$,  $\ant (\TT_A)$ and $\ant (\TT^A)$, for a given  normal  matrix $A=F(d,\su d,\su h,g)$?
\begin{itemize}

\item If $h_1=h_2=h_3=0$, then  $A$ is  idempotent normal.
 By lemma \ref{lem:A=A2},  $\TT_{A}$ is  good, so that \emph{the columns of $-A^T$ are precisely the vertices of the sides of $\TT_{A}$}, see figure \ref{fig:normal_idempot_11}. By the max--min duality, \emph{the columns of $A$ are  the stable intersection points of the tropical lines $\su {A}$}. Therefore,
\begin{equation}
\soma(\TT_A)=\TT_{A}=\span(A)=\cospan(-A^T)=\TT^{A}=\soma(\TT^A).\label{eqn:span-cospan}
\end{equation}

\item If $h_{j+1}>0$ and $d_j=0$ for some $j$, by the previous item,
\begin{equation}
\soma(\TT_A)=\TT_{A^{\odot2}}=\TT^{A^{\odot2}}=\soma(\TT^A),\label{eqn:soma}
\end{equation}
even though $\TT_A\neq \TT^A$.

Write $S=\soma(\TT_A)$. Each antenna $a$ of $\TT_A$ grows at a vertex $v$ of $S$. An antenna $a'$ of $\TT^A$ grows at $v$, as well and, in fact, the integer length of $a$  and $a'$ is the same.
We claim that \emph{there exists a unique two--dimensional peripheral cell, denoted $P_a$, in the cell decomposition $\CC^A$ such that $a'\cup a\subset \overline {P_a}$.}
Indeed, by hypothesis, $A=F(d,\su d,\su h,g)$ and $A^{\odot2}=L(d,\su d)$, for some $d,\su d, \su h,g\ge0$ with $h_{j+1}>0$ implying $d_j=0$ and $g>0$ implying $d=d_1=h_3=0$.
Say $h_2>0$, $d_1=0$ so that $\TT_A$ has an antenna $a$ pointing north--east (see figure \ref{fig:overlapping_20}). Looking at (\ref{eqn:LLLcero}), (\ref{eqn:FHcero}) and (\ref{eqn:FHTmenoscero}) we know that the end points of $a$ are $v=\col((A^{\odot2})_0,2)=[d,2d+d_2,0]^T$ and  $\col(A_0,2)=[d+h_2, 2d+d_2+h_2,0]^T$ while the end points of $a'$ are $v$ and  $\col((-A^T)_0,3)=[d, 2d+d_2+h_2,0]^T$. Therefore, we have the following expression of $P_a$ in $Z=0$:
\begin{equation}
P_a=\{(x,y)^T\in\T^2: d<x,\quad 0<y-x-d-d_2<h_2\}.\label{eqn:cell_P_a}
\end{equation}

\item If  $A=F(0,0,d_2,d_3,0, h_2,0,g)$, with $d_2,d_3\ge0$, $g>0$, then  $\TT_A$ has an antenna $a$ of integer length  $g$
emanating from $\col(A^{\odot2},3)$. Computing $(-A^T)_0$ we see that $\TT^{A}$ has an antenna $a'$ of the same integer length, emanating from the same point. The corresponding peripheral cell $P_a$ is expressed as follows in $Z=0$:
\begin{equation}
P_a=\{(x,y)\in\T^2: -d_3-g<x<-d_3,\  y<x\}.\label{eqn:cell_P_a_g}
\end{equation}

\end{itemize}
\begin{figure}[H]
 \centering
  \includegraphics[width=10cm,keepaspectratio]{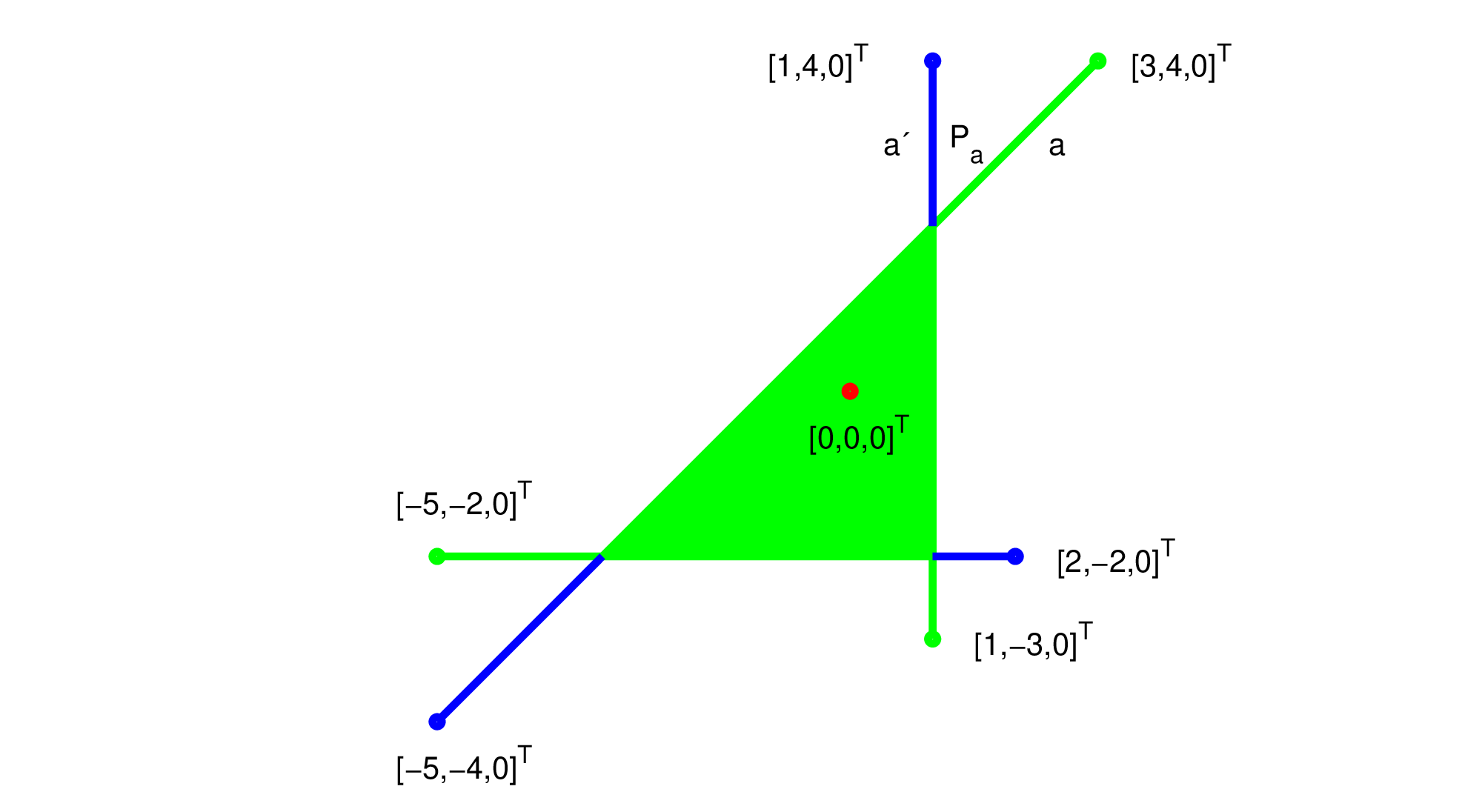}\\
  \caption{Tropical triangles $\TT_A$, $\TT^A$, antennas $a$ and $a'$ and the two--dimensional cell $P_a$, for $A$ in (\ref{eqn:cell_decom_bad_9}).}
  \label{fig:overlapping_20}
  \end{figure}

\section{The map $f_A$ for a $A=F(d,\su d,\su h,g)$}
Recall that $\overline{\Gamma}$ denotes the closure of $\Gamma\subseteq \T\P^2$. The following theorem fully describes the action of the map $f_A$  on each point of $\T\P^2$, since $\T\P^2$ is a (finite) union of $\overline{\Gamma}$, where $\Gamma$ is a two--dimensional cell in $\CC^A$. The cell $\Gamma$ is of the following types: either $B^A$ (central cell), or $P_a$ (peripheral cell parallel to some coordinate axis, associated to some antenna $a$), or $P$ (peripheral cell parallel to some coordinate axis, not associated to any antenna) or $C$ (corner). The cell $B^A$ can be empty. If $A=F(d,\su d,\su h,g)$, then $B^A=\emptyset$ if and only if $d=d_j=d_{j+1}=0$, for some $j$ modulo 3.

\begin{thm}\label{thm:f_A}
 Let $A=F(d,\su d,\su h,g)$ be given. Then
 \begin{enumerate}
 \item $f_A|_{\overline{C}}$ collapses to some vertex of $\TT_A$, for every corner $C$ in $\CC^A$,\label{cond:cuatro}
\item $f_A|_{\overline{P_a}}=a$, for every antenna $a$ of $\TT_A$,  \label{cond:dos}
\item  $f_A|_{\overline{P}}$ is the classical projection onto  $\TT_A$, in  the direction of $P$, for every two--dimensional  peripheral cell $P$ of $\CC^A$, if $P\neq P_a$, for every $a$ antenna of $\TT_A$,\label{cond:tres}
\item $\soma(\TT_A)$ is the set of fixed points of $f_A$ and, if $B^A\neq\emptyset$, then $\soma(\TT_A)=\overline{B^A}$.\label{cond:uno}

\end{enumerate}
 \end{thm}

\begin{proof} $f_A$ is continuous, so that $f_A|_{\overline\Gamma}$ is easily computed from $f_A|_{\Gamma}$.
Part (\ref{cond:cuatro}) was proved in p. \pageref{dfn:corners}.
\begin{itemize}
\item Suppose that $h_1=h_2=h_3=0$. Then $A=L(d,\su d)$ is
idempotent normal. Since $\TT_A$ is a tropical triangle without antennas, then part (\ref{cond:dos}) does not apply.  In order to prove part (\ref{cond:tres}), let us work in $Z=0$. We have
 $$A_0=\LLLcero,$$
$$(-A^T)_0=\LLLTmenoscero.$$
Let  $P$ be a peripheral cell in $\CC^A$; say, $P$ is   parallel to the $Y$ direction. Then either
\begin{equation}
P=\{(x,y)^T\in \T^2:-2d-d_3<x<d+d_1-d_3,\quad y<-d-d_3\}
\end{equation}
or
\begin{equation}
P=\{(x,y)^T\in \T^2:d+d_1-d_3<x<d+d_1,\quad y<x-2d-d_1\}.
\end{equation}
 In the former case, $f_A(p)=[x,{-d-d_3},0]^T$;  in the latter, $f_A(p)=[x,{x-2d-d_1},0]^T$.
 We can be  more concise: if $p=[x,y,0]^T\in P$, then either $y=-\infty$ or $y\in\R$ is  big and negative. Therefore,  $f_A(p)=A\odot p$ is a tropical linear combination \emph{only of} $c_1=\col(A,1)$ and $c_3=\col(A,3)$; in particular, $f_A(p)$   belongs to the tropical segment $\span(c_1,c_3)\subseteq\TT_A$ (and $f_A(p)$ does not depend on $y$).
 This proves part (\ref{cond:tres}).

By lemma \ref{lem:normal}, each point in  $\span(A^{\odot2})=\soma(\TT_A)$  is fixed by $f_A$. From equality (\ref{eqn:B^A}) and parts
 (\ref{cond:tres}), (\ref{cond:dos}) and (\ref{cond:cuatro}), now part (\ref{cond:uno}) follows.

\item Suppose that $h_{j+1}>0$ and  $d_j=0$, for some $j$; say $j=2$ and $P_a$ is given in (\ref{eqn:cell_P_a}). For any $p=[x,y,0]^T\in\overline{P_a}$ we have $f_A(p)=[y-d-d_2,y,x-d]^T=[y-x-d_2,y-x+d,0]^T\in a$. In particular,  $f_A(a')=a$, where $a'$ is the antenna  of $\TT^A$ corresponding to $a$. This  proves part (\ref{cond:dos}). Parts (\ref{cond:tres}) and (\ref{cond:uno}) are proved as in the previous case.

\item Suppose that $g>0$ and $d=d_1=h_3=0$. Then $P_a$ is given in (\ref{eqn:cell_P_a_g}). The proof is similar to the previous case.
\end{itemize}
\end{proof}

\begin{rem}
If $A=L(d,\su d)$  for $d,\su d\ge0$, then $f_A=f_A\circ
f_A$, by theorem \ref{thm:f_A}, meaning that $f_A$ is some
kind of projection. But, in general, $f_A$ is different
from  the projector map onto $\TT_A$ in  p.
\pageref{dfn:projector}.  For instance,  consider the
matrix $A=L(3,9,2,4)$, i.e.,
$$A=\left[\begin{array}{ccc} 0&-5&-10\\ -15&0&-7\\ -12&-8&0 \end{array}\right],\quad A_0= \left[\begin{array}{ccc} 12&3&-10\\-3&8&-7\\0&0&0 \end{array}\right],\quad
p=\co{-12}00.$$
Then $$\rho(p)=\co{-12}{-27}{-24}\oplus\co{-12}{-7}{-15}\oplus\co{-12}{-9}{-2}=\co{-12}{-7}{-2}=\co{-10}{-5}{0}$$
$$f_A(p)=\co{-5}{0}{0}\neq \rho(p).$$
\end{rem}

 The results above extend to two types of matrices over $\T$.
\begin{itemize}
\item If $A$ is a permutation matrix, then  $I=L(0,\infty,\infty,\infty)$ can be conceived as the lower canonical normalization of $A$;  $I$ can be viewed as the limit of $L(0,\su d)$, as all the $d_j$'s tend to infinity. In this case, $\TT_A=\TT_I=\T\P^2$ and $f_A=f_I$ is the identity.

\item If $A=D\odot P$  for some permutation matrix $P$ and
\begin{equation}\label{eqn:D}
D=\left[\begin{array}{ccc}
-\infty&0&0\\0&-\infty&0\\0&0&-\infty\\
\end{array}\right],
\end{equation}
 then  $N=F(0,0,0,0,\infty,\infty,\infty,0)$ can be considered the lower canonical normalization of $A$ and $N$ is, in some sense, the limit of $F(0,0,0,0,\su h,0)$, as all the $h_j$'s tend to infinity; see (\ref{eqn:alpha}). In this case, $\TT_A$ is just the tropical line $L_0$ and $\soma(\TT_A)$ is the origin.
\end{itemize}

Neither $-I^T$  nor $-D^T$ can be written with entries in $\T$. However, according to the equivalence relation in p. \pageref{asser:equiv},  $[0,a,b]$ tends to  $[-\infty,0,0]$ (resp. $[-\infty,0,-\infty]$) when $a,b\to \infty$ (resp. $a\to \infty$ and  $b$  remains fixed), so that
$D$ can be interpreted as $-I^T$.

If $A$ is a permutation matrix or $A=D\odot P$, then the cell decomposition $\CC^A$ is obvious.

\begin{dfn}\label{dfn:admisible}
A   $3\times3$ matrix $A$ is \emph{admissible} if either $A$ is real or $A=B\odot P$, where $B=D$  in (\ref{eqn:D}) or $B=I$ and $P$ is a permutation matrix.
\end{dfn}

Finally, we can   describe the map $f_A$, for any $3\times 3$  admissible matrix $A$. First, we find the lower canonical normalization $N=F(d,\su d,\su h,g)=P\odot A\odot Q$ to obtain $f_N=f_P\circ f_A\circ f_Q$; then we apply theorem \ref{thm:f_A}, knowing that $f_P$ and $f_Q$ are just changes of coordinates.

Consider
the set $S_A$ of points where $f_A$ is injective, i.e.,
$$S_A=\{q\in\T\P^2: \exists  \ !\ p\in\T\P^2, f_A(p)=q\}.$$  On $\R^n$,  the set $S_A$ plays an important role in \cite{Butkovic_S}.

Suppose that $A$ is admissible. If $A=P$ is a permutation matrix, then $S_A=\T\P^2$. If $A=D\odot P$, then $S_A$ reduces to the origin. If $A$ is real and $A=F(d,\su d,\su h,g)$, then  $S_A= B^A$, by theorem \ref{thm:f_A} and equality (\ref{eqn:B^A}). And if $N=P\odot A\odot Q$ is the lower canonical normalization of $A$, then $S_A=S_{A\odot Q}=f_{P^{\odot-1}}(B^N)=B^A$. In other words,  \emph{$f_A$ is injective precisely on $B^A$, when $A$ is real}.

%
%
%

\begin{cor}
If $A$ is an admissible matrix, then $f_A$ transforms tropical collinear points into tropical collinear points.
\end{cor}
\begin{proof}
All we need to show is that the image of a tropical line $L\subset \T\P^2$ is contained in some tropical line $L'$. Now, the cell decomposition $\CC^A$ on $\T\P^2$ induces a  cell decomposition on $L$. The triangle $\TT_A$ is also decomposed into cells. Now, a case by case analysis (depending on the position of the vertex of $L$) and  theorem \ref{thm:f_A}, shows that each cell $\Gamma$ of $L$ is mapped by $f_A$ inside a cell $\Gamma'$ of $\TT_A$. And the union of all such $\Gamma'$'s is contained in a tropical line $L'$.
\end{proof}

\section*{Acknowledgements}
I am very grateful to two anonymous  referees for their very valuable suggestions.
Besides, I would like to thank  my former students Fernando Barbero and  Elisa Lorenzo for their interest and support.

\end{document}